\documentclass[leqno,11pt]{amsart}
\usepackage{amsmath,amsthm,amscd,amssymb,amsfonts,amsbsy}
\usepackage[colorlinks,citecolor=red,hypertexnames=false]{hyperref} 
\usepackage[numbers]{natbib} 
\usepackage{enumerate,graphicx}

\allowdisplaybreaks

\numberwithin{equation}{section}

\theoremstyle{plain}
\newtheorem{theorem}{Theorem}[section]
\newtheorem{lemma}[theorem]{Lemma}
\newtheorem{corollary}[theorem]{Corollary}
\newtheorem{proposition}[theorem]{Proposition}

\theoremstyle{definition}
\newtheorem{definition}[theorem]{Definition}

\newcommand{\RR}{{\mathbb{R}}}

\newcommand{\ZZ}{{\mathbb{Z}}}

\newcommand{\eps}{\varepsilon}

\newcommand{\vp}{\varphi}

\renewcommand{\d}{\partial}

\newcommand{\sm}{\setminus}

\renewcommand{\tilde}{\widetilde}
\newcommand{\wh}{\widehat}
\renewcommand{\bar}{\overline}

\def\div{\mathop{\operatorname{div}}}

\usepackage{mathtools}

\mathtoolsset{showonlyrefs}

\begin{document}

\author[D. N. Arnold]{Douglas N. Arnold}
\email[]{arnold@umn.edu}
\address{School of Mathematics, University of Minnesota, Minneapolis, MN, USA}

\author[G. David]{Guy David}
\email[]{guy.david@u-psud.fr} 
\address{Univ Paris-Sud, Laboratoire de Math\'ematiques, CNRS, UMR 8658 Orsay, F-91405}

\author[M. Filoche]{Marcel Filoche}
\email[]{marcel.filoche@polytechnique.edu}
\address{Physique de la Mati\`ere Condens\'ee, Ecole Polytechnique, CNRS, Palaiseau, France}

\author[D. Jerison]{David Jerison}
\email[]{jerison@math.mit.edu}
\address{Mathematics Department, Massachusetts Institute of Technology, Cambridge, MA, USA}

\author[S. Mayboroda]{Svitlana Mayboroda}
\email[]{svitlana@math.umn.edu}
\address{School of Mathematics, University of Minnesota, Minneapolis, MN, USA}

\thanks {
Arnold is partially supported by the NSF grant DMS-1719694 and Simons Foundation grant 601937, DNA.
David is supported in part by the ANR, programme blanc GEOMETRYA ANR-12-BS01-0014,
EC Marie Curie grant MANET 607643, H2020 grant GHAIA 777822, and Simons Foundation grant 601941, GD.
Filoche was supported in part by Simons Foundation grant 601944, MF.
Jerison was partially supported by NSF Grants  DMS-1069225 and DMS-1500771, a Simons Fellowship,
and Simons Foundation grant 601948, DJ.
Mayboroda is supported in part by the Alfred P. Sloan Fellowship, the NSF INSPIRE Award DMS 1344235, NSF CAREER Award DMS 1220089, a Simons Fellowship, and Simons Foundation grant 563916, SM.
Part of this work was completed during Mayboroda's visit to Universit\'e Paris-Sud, Laboratoire de Math\'ematiques, Orsay, and Ecole Polytechnique, PMC, and we thank the corresponding Departments and Fondation Jacques Hadamard for support and hospitality.
}

\title[Localization of eigenfunctions]
{Localization of eigenfunctions via an effective potential}

\begin{abstract}
We consider the localization of eigenfunctions for the operator $L=-\div A \nabla + V$ on a Lipschitz domain $\Omega$ and, more generally, on manifolds with and without boundary. In earlier work, two authors of the present paper demonstrated 
the remarkable ability of the landscape, defined as the solution to $Lu=1$, to predict the location of the localized eigenfunctions.   Here, we explain and justify a new framework
that reveals a richly detailed portrait of the eigenfunctions and eigenvalues.   
We show that the {\em reciprocal} of the landscape function, $1/u$, acts as 
an {\em effective potential}.    Hence from the single measurement of $u$,
we obtain, via $1/u$, explicit bounds on the exponential decay of 
the eigenfunctions of the system and estimates on the distribution of eigenvalues 
near the bottom of the spectrum. 
\end{abstract}

\date{\today }

\maketitle

\section{Introduction}\label{intro}

The term localization refers to a wide range of phenomena in mathematics and condensed matter physics in which eigenfunctions of an elliptic system concentrate on a small portion of the original domain and nearly vanish in the remainder, hindering or preventing wave propagation.  For many decades, its different manifestations have been a source of wide interest, with an enormous array of applications.   In addition to celebrated results
concerning localization by disordered potentials \cite{50years, An, AM, FS, BK, EA}, 
there is localization by randomness in the coefficients of $-{\rm div} A \nabla$ and of the Maxwell system \cite{FK1, FK2}, localization by a quasiperiodic potential \cite{JL}, and localization by fractal boundaries \cite{F+2007}, to mention only a few examples.
However, with the notable exception of the recent work \cite{JL} for a 1D almost Matthieu operator, these results do not address detailed, deterministic
geometric features of the localized eigenfunctions. 

The present paper changes the point of view through the introduction of a new {\em effective potential}, 
and applies it to establish the location, shape,  and a detailed structure of the exponential decay of
the eigenfunctions of the operator $L = -\div  A \nabla  + V$ on a finite domain, as well as estimates on 
its spectrum.

In 2012, Filoche and Mayboroda introduced the concept of the {\it landscape}, namely the solution $u$ to $Lu=1$ for an elliptic operator $L$, and showed that this single
function  has  remarkable power to predict the shape and location of localized low energy eigenfunctions of $L$, whether the localization is triggered by the disorder of the potential, the geometry of the domain, or both (see \cite{FM-PNAS}). These ideas led 
to beautiful new results in mathematics \cite{St, LS}, as well as theoretical and experimental physics \cite{PRL-3}.   

In this paper and its companion papers \cite{PRL-2} in physics and \cite{SP}
in computational mathematics,  we propose a new framework that greatly extends 
the predictive  power of the landscape function $u$. 
We show that the reciprocal $1/u$ of the landscape function should be viewed as
an {\em effective quantum potential} revealing detailed structure of the eigenfunctions.
The eigenfunctions of $L$ reside 
in the wells of $1/u$ and decay exponentially across the barriers of $1/u$.  
Under hypotheses on the behavior of $u$ that can be confirmed easily
and efficiently numerically, the original domain splits into independently vibrating regions,
and  the global eigenfunctions are exponentially close to eigenfunctions of subregions.  
As a corollary, we prove an approximate diagonalization of the operator
and confirm that localization according to $1/u$ gives an accurate eigenvalue count up to exponential
errors.

Predicting the eigenvalue count or ``density of states" is an important goal linking
this paper to the other two.    The proposal in \cite{PRL-2} to
use $1/u$ to estimate the density of states, starting from the very bottom
of the spectrum, has provoked a burst of applications beyond the scope
of the single-particle Schr\"odinger equation. 
In particular, in the context of the Poisson-Schr\"odinger system, the paper
\cite{FPW+}  finds an iterative algorithm
that speeds up the time it takes to compute the performance of the type of semiconductor 
used in LED devices from one year to one day.   
The key to this acceleration is that at each step of the iteration, a new potential is computed as a function of the density of states.   This modifies in turn the operator $L$ and therefore the effective potential $1/u$
from which the next density of states is derived, without ever solving the Schr\"odinger equation.
In the companion article \cite{SP} in computational mathematics,
we explore systematically efficient shortcuts leading from the effective potential to the density of states.

Although some of our applications are to random regimes,
the effective potential $1/u$ is a deterministic tool.  It is
not designed to replace probabilistic methods, but to complement 
and enhance them by providing a new way to detect the quantum geometry 
of disordered materials.   Statistical mechanics  often treats the source of disorder 
as a black box, whereas this mechanism allows us to enter the box and identify 
detailed deterministic features of the disorder.

The paper is organized as follows. In Section \ref{sec:outline}, we state our results in a
special case and illustrate their numerical significance.  
In Section \ref{sec:prelim}, we give our main definitions and state some preliminary estimates on the landscape function and eigenfunctions. In Section \ref{sec:agmon}, we derive our exponential decay estimates, known
as Agmon estimates, in the setting of bounded domains in $\RR^n$.  In Section \ref{sec:diagonalization},
we deduce the approximate diagonalization into localized eigenfunctions and estimates on the eigenvalue distribution.
Finally, in Section \ref{sec:manifolds},
we describe how to generalize our theorems to manifolds and prove the boundary regularity theorems stated in Section \ref{sec:prelim}. We also address the difficulty that  Agmon metrics are only defined for continuous coefficient matrices $A$; because our estimates are independent of the modulus of continuity, we are able to use a fairly straightforward procedure to approximate bounded measurable coefficient matrices by continuous ones.

\section{Outline of Results and Comparison with Numerical Examples}\label{sec:outline}

To describe our results we consider the very special case in which the operator is (minus) the ordinary Laplace operator plus a nonnegative, bounded potential,
\[
L = - \Delta + V \quad (0 \le V(x) \le \bar V; \quad \bar V  : = \sup V)
\]
acting on periodic functions, that is, on the manifold $M = \RR^n/T \ZZ^n$.  
 It is
crucial to applications that the estimates be independent of the ``size" $T$
of the manifold $M$ as $T\to\infty$.   What makes them even more valuable is that
they are essentially universal, as we shall discuss later in this section.\footnote{Furthermore, in the body of the paper, we will treat operators with bounded measurable coefficients on Lipschitz and more general domains and on compact $C^1$ manifolds with and without boundary; see Sections \ref{sec:prelim} and  \ref{sec:manifolds}.}

Assume that $V$ is positive on a set of positive measure. Then the landscape function $u$, the solution to $Lu=1$ on $M$, exists and is unique.  Moreover, $u>0$ by the maximum principle. 
Our starting point is the conjugation of the operator $L$ by multiplication
by $u$:
\[
\tilde L g := \frac1u L(gu) = -\frac1{u^2} \div (u^2  \nabla g) + \frac1u g.
\]
The operator $\tilde L$ has a similar form to  $L$ but with the new
potential $1/u$ replacing $V$. 
Writing the quadratic form associated with the operator
$L$ in terms of $\tilde L$, we find the identity
(Lemma~\ref{lem:identity})
\begin{equation} \label{eq:introidentity}
\int_M [|\nabla f|^2 + Vf^2]\, dx 
= \int_M 
\left( u^2 |\nabla (f/u)|^2+ \frac1u \, f^2 \right)\, dx,
\end{equation}
which holds for all $f\in W^{1,2}(M)$. In particular, 
\begin{equation}\label{eq:formbound}
\int_M [|\nabla f|^2 + Vf^2]\, dx \ge \int_M (1/u) f^2 \, dx.
\end{equation}
Inequality \eqref{eq:formbound}  suggests that we can 
replace $V$ with a new {\em effective potential function} $1/u$.  In fact,
we will need the full identity  \eqref{eq:introidentity} to demonstrate
this.  The identity reflects a trade in  kinetic and potential energy, enabling
$1/u$ to capture effects of both the kinetic term $|\nabla f|^2$ and the potential term $Vf^2$ rather than only the potential energy.

An example of the localization we are trying to predict and control is shown in Figure
1, which 
depicts a Bernoulli potential $V$ on $\RR^2/T\ZZ^2$ with $T=80$ and constant values on unit squares, 
$V =0$ on white squares and $V=4$ on black squares.   The values were chosen
independently, with probability $30$\% for $V=4$ and $70$\% for $V=0$. 
At the right is the graph of fifth eigenfunction.   In spite of the fact that the zero
set of $V$ percolates everywhere, this eigenfunction and dozens of others
are highly localized.

\begin{figure}[htbp]   \label{fig:potential+spike}   
\begin{tabular}{cc}
\hspace{-40pt}
\includegraphics[height=2in]{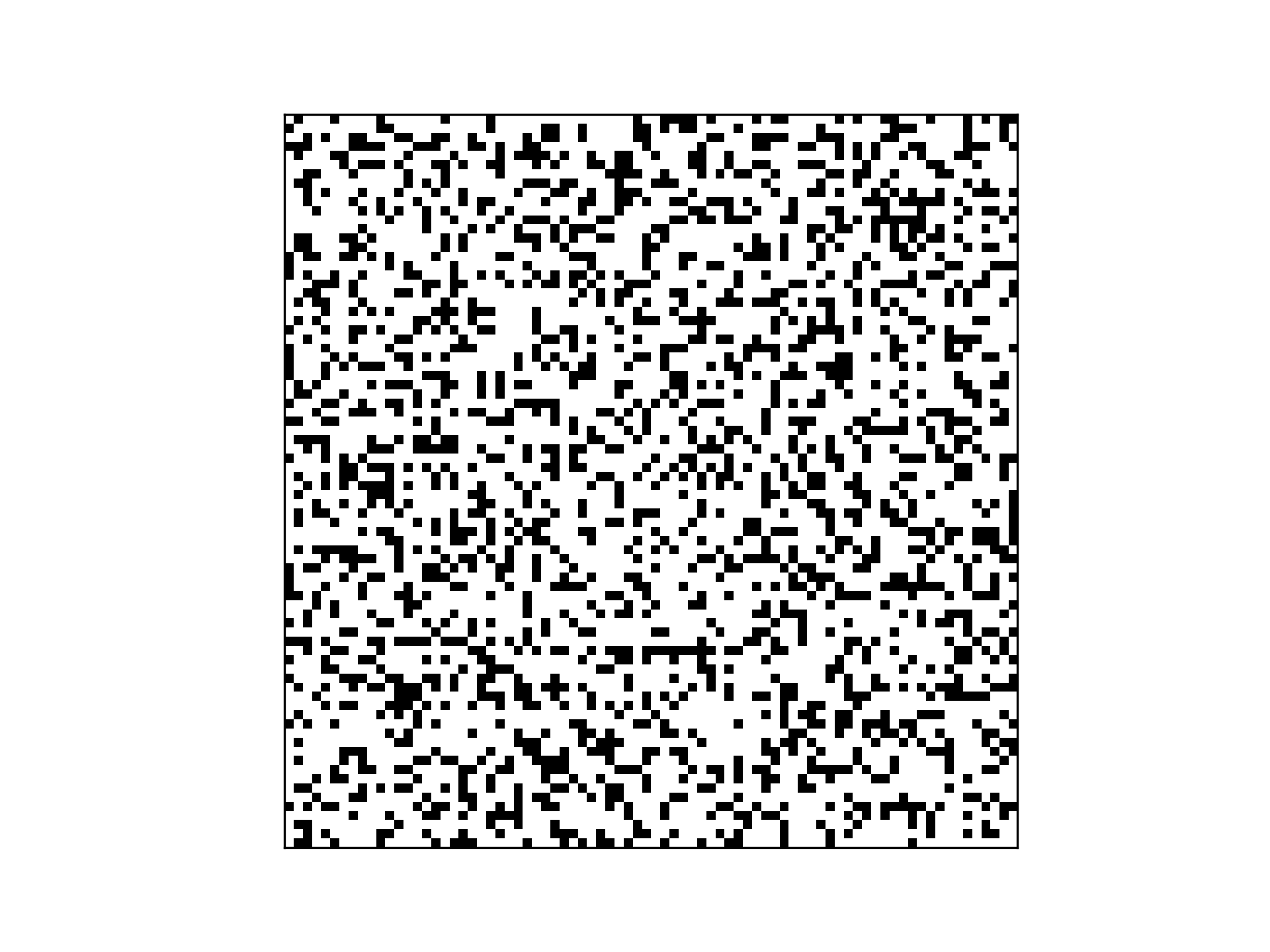} 
\hspace{-50pt}
\includegraphics[height=2in]{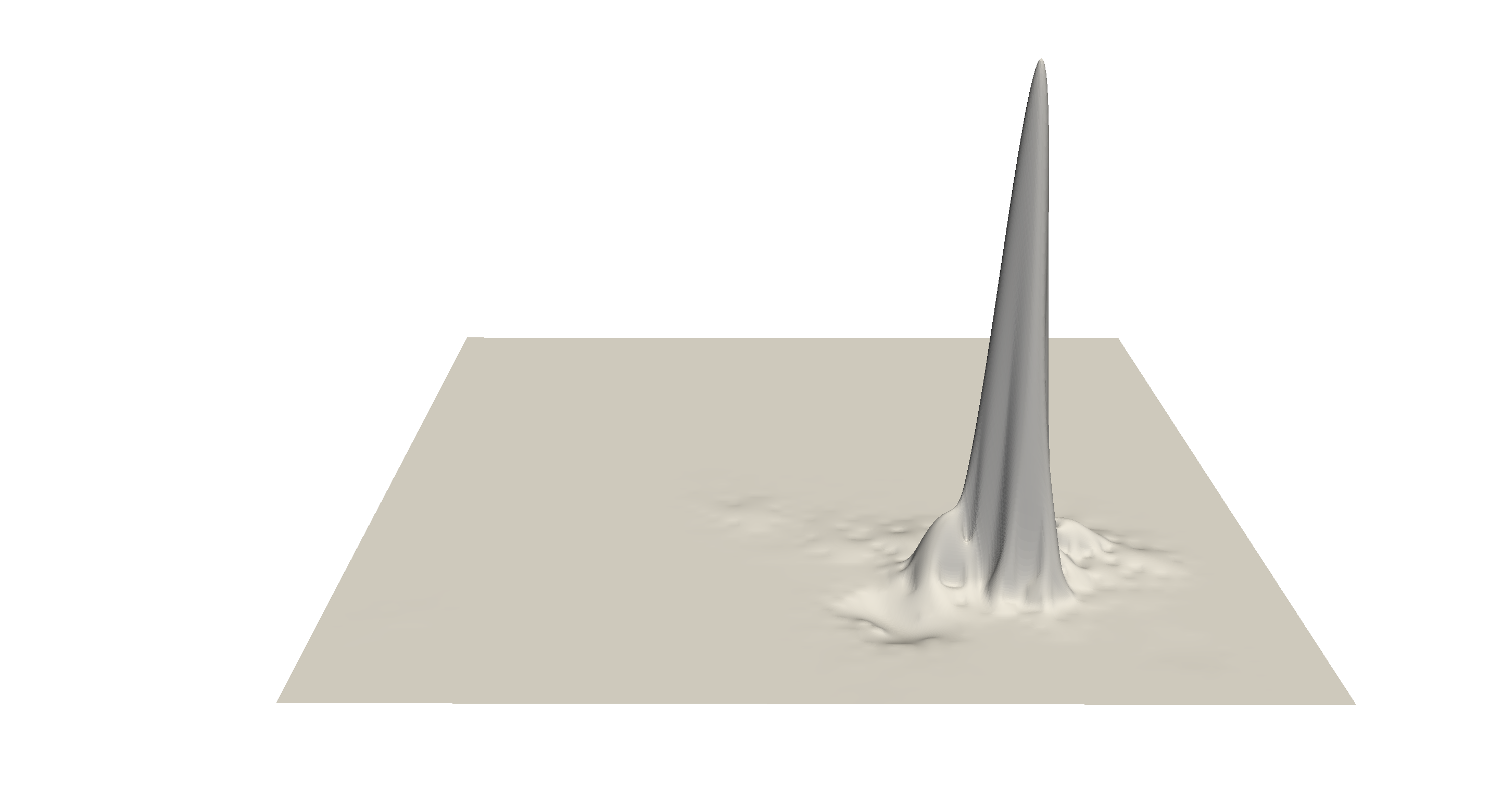} 
  \end{tabular}
     \caption{Bernoulli potential (left) and the fifth eigenfunction (right).}
\end{figure}

\begin{center}
{\bf Exponential decay}
\end{center}

The first main result of this paper is the rigorous proof that the steep
decay in Figure 1 
comes from the barriers of the effective potential.
We do this by formulating and proving appropriate exponential decay
estimates of  Agmon type (see \cite{AgmonBook,DFP}). 
Roughly speaking,
these theorems say that if \eqref{eq:formbound} holds, then eigenfunctions of eigenvalue $\lambda$ have ``most'' of their mass in the  region 
\[
E(\lambda + \delta) = \{x\in M: 1/u(x) \le \lambda + \delta\}
\]
for a suitable small $\delta>0$, and exponential decay in the complementary region.   

To formulate our estimate precisely,
consider the weights
\[
w_\lambda(x): = \max \left(\frac1{u(x)} - \lambda, \, 0\right). 
\]
Exponential decay is expressed in terms of the so-called Agmon distance, 
traditionally built from $V$, but for our purposes arising from $1/u$.  
We define our version of Agmon distance, which we will refer to loosely as the
{\em effective distance}, as the degenerate metric 
on $M$ given by
\[
\rho_\lambda(x,y) = 
\inf_{\gamma} \int_0^1 w_\lambda(\gamma(t))^{1/2}\, |\dot \gamma(t)|\, dt,
\]
with the infimum taken over absolutely continuous paths $\gamma: [0,1]\to M$
from $\gamma(0)=x$  to $\gamma(1) = y$.  

\begin{theorem}\label{thm:introexpdecay} (see Corollary~\ref{cor:expdecay}) 
Let $\psi$ be an eigenfunction: $L\psi = \lambda \psi$ on $M$. Let
\[
h(x) = \inf \{\rho_\lambda (x,y): y\in E(\lambda +\delta)\}
\]
be the effective distance from $x$ to $E(\lambda+ \delta)$. Then 
\begin{equation}\label{eq:introdecay}
\int_{\{h\ge 1\}} e^h (|\nabla \psi|^2 + \bar V \psi^2) \, dx 
\le 50 (\bar V/\delta)  \int_M \bar V \psi^2 \, dx.
\end{equation}
\end{theorem}
The theorem says that the square density and energy of the eigenfunction are at most of size $e^{-h}$,
with $h = h_{\lambda,\delta}$ the effective distance from $E(\lambda + \delta)$.
The main difficulty of the proof is to compensate for the price we
paid for replacing $V$ with $1/u$, namely that the gradient term
$|\nabla f|^2$ has been replaced by $u^2 |\nabla(f/u)|^2$ in \eqref{eq:introidentity}.  
We  can't afford this dependence on $u$, and a crucial feature of the 
estimate we obtain in \eqref{eq:introdecay} is that this part of the dependence on $u$ disappears, leaving
only the effects of $1/u$.

Remarkably, we get a uniform bound, independent of the dimension $n$ 
and the  size $T$ of the manifold.  It is universal in that it depends only on 
the effective distance and 
the scale-invariant ratio  $\delta/\bar V$, where $\delta$ is 
a spectral gap.    As such, it can be interpreted easily both numerically and physically
across a wide family of contexts.

\begin{figure}[htbp] \label{fig:effective+eigen5}  
\begin{tabular}{cc}
\hspace{-40pt}
\includegraphics[height=2in]{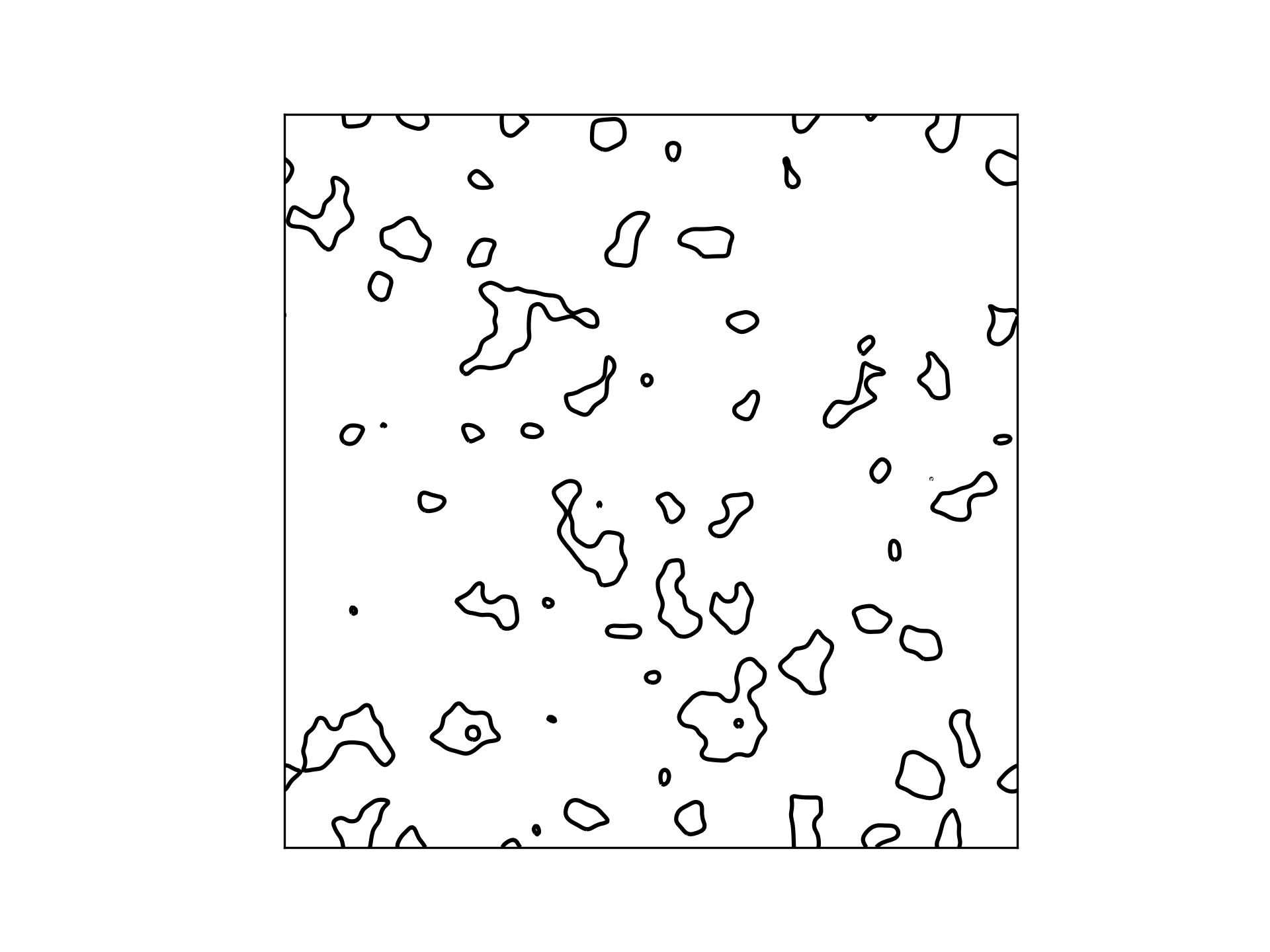} &
\hspace{-40pt}
\includegraphics[height=2in]{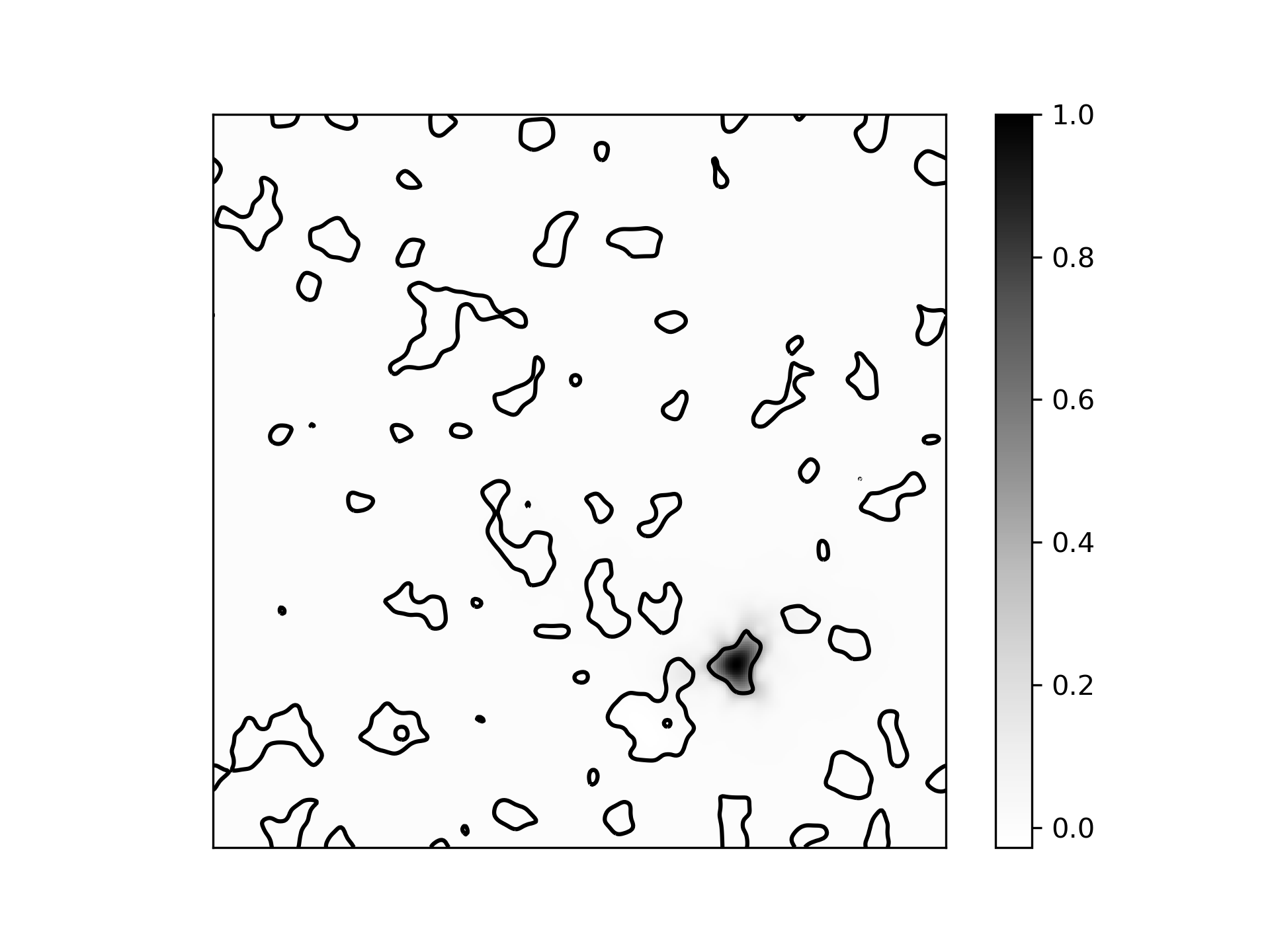} 
  \end{tabular}
     \caption{$E(\lambda_5+\delta) = \{1/u(x) \le \lambda_5  + \delta\}$ (left)
  with fifth eigenfunction superimposed in grey scale (right).}
\end{figure}

To illustrate this exponential decay, we compute the effective potential
$1/u(x)$ for the Bernoulli potential in Figure 1.
Figure 2 
shows the contour of $E( \lambda_5 + \delta)$ on the left with $\lambda_5 = 0.45508$, the fifth eigenvalue.
(The value $\delta = 0.005$  was chosen as the average spacing between eigenvalues 
in the vicinity of the fifth.)
Overlaid on the right in grey scale are the values of fifth eigenfunction $\psi_5$.
Note that  most of $\psi_5$ occupies just one component of the set 
$E(\lambda_5 + \delta)$.   In fact, dozens of eigenfunctions 
coincide essentially with single components or clusters of components.
\bigskip

\begin{center}
\bf Approximate diagonalization
\end{center}

So far, estimate \eqref{eq:introdecay} only guarantees that $\psi_5$ is
supported primarily in a union of wells, i.~e., it is mostly a linear combination of highly
localized functions,
whereas Figures 1 and 2 show that the eigenfunction is primarily a single spike.   
We want to show that eigenfunctions are single spikes or clusters of 
spikes and  justify implicitly
the numerical procedure for finding the eigenfunctions
in order by examining the wells separately, starting from the deepest (see \cite{SP}).

To prove that eigenfunctions localize to a single well or a cluster, we establish
an approximate diagonalization.    This will require an extra assumption
on spectral gaps.   For the purposes of localization and diagonalization, near
multiplicity, or  resonance, is the enemy.  Eigenfunctions with
nearly the same eigenvalue can, in fact, share wells.  

 We introduce a space of
localized eigenfunctions as follows.
Consider a threshold $\bar \mu$ that will be used
to handle eigenvalues $\lambda \le \bar \mu -\delta$.  
Choose any  subdivision\footnote{The $E_\ell$ are typically connected components of $E$, but since the theorem
is stronger when the minimum separation $\bar S$ is larger, it is sometimes useful to merge nearby 
wells into one set $E_\ell$.}   of $E = E(\bar \mu + \delta)$ into a finite collection of disjoint closed subsets
\[
E
= \bigcup_\ell E_\ell\, .
\]
Let $\bar S$ denote the smallest effective distance  $\rho_{\bar \mu}$ between distinct pairs of sets $E_\ell$ and $E_{\ell'}$.   Let  $\Omega_\ell$ be the $\bar S/2$ neighborhood\footnote{The sets $\Omega_\ell$
can also be chosen to be somewhat larger, provided each $\Omega_\ell$ is separated by at  least $\rho_{\bar\mu}$ distance $\bar S/2$ from $E_{\ell'}$ for every $\ell'\neq \ell$. They are roughly in the spirit of Voronoi cells.}
 of $E_\ell$ in the effective distance $\rho_{\bar \mu}$.
Let $\vp_{\ell,j}$, $j=1, \dots,$ be the orthonormal basis of $L^2(\Omega_\ell)$ of eigenfunctions of $L$ satisfying the Dirichlet condition $\vp = 0$ on $M \sm \Omega_\ell$. By results analogous to the exponential bounds for $\psi$, these functions $\vp_{\ell,j}$ 
are concentrated near $E_\ell$ and decay exponentially in the larger region $\Omega_\ell$, provided the corresponding eigenvalue satisfies
$\mu_{\ell,j} \le \bar \mu$.  In other words, such functions are {\em localized to
a single well or cluster $E_\ell$}  in 
$M$.

Denote by $\Phi_{(a,b)}$ the orthogonal projection onto the subspace of $L^2(M)$ spanned by $\vp_{\ell,j}$ with eigenvalues between $a$ and $b$, and $\Psi_{(a,b)}$ the corresponding spectral
projection for eigenfunctions of $L$.  Denote by $\|\cdot \|$ the norm of $L^2(M)$.
Our main result is the following.
\begin{theorem}\label{thm:introspectralprojection} (see Theorem~\ref{thm:spectralprojection})
If  $\psi$ is an eigenfunction of $L$ with eigenvalue $\lambda$ on $M$ and 
$\lambda\le \bar \mu -  \delta$, then
\begin{equation}\label{eq:introdiag}
\|\psi - \Phi_{(\lambda-\delta,\lambda+ \delta)} \psi  \|^2 
\le 300 \left(\frac{\bar V}{\delta}\right)^3 e^{-\bar S/2} \|\psi\|^2  .
\end{equation}
If $\vp = \vp_{\ell,j}$ is a localized eigenfunction with eigenvalue $\mu = \mu_{\ell,j} \le \bar \mu - \delta$, 
then 
\[
\| \vp - \Psi_{(\mu-\delta, \mu + \delta)}\vp \|^2 \le 300 \left(\frac{\bar V}{\delta}\right)^3 e^{-\bar S/2} \|\vp\|^2.
\]
\end{theorem}
The interpretation is that the eigenfunctions $\psi$ are linear combinations of
localized $\vp_{\ell,j}$ associated with the spectral band  $\lambda  \pm \delta$.
In particular, if the projection has rank one, then $\psi$ lives primarily in
one well or cluster $E_\ell$.
This is the kind of localization we see in numerical simulation. 

Let us make the spectral gap condition required for the projection to
have rank one more explicit.     If we choose $\delta$ so that 
\begin{equation}\label{eq:spectralgap}
\delta/\bar V >> e^{-\bar S/6} ,
\end{equation}
then the constant on the right side of \eqref{eq:introdiag} is much smaller than $1$.
If there is only one eigenvalue $\mu_{\ell,j}$ 
in the range $(\lambda-\delta, \lambda+\delta)$,  then the projection has rank one, 
and the eigenfunction $\psi$ is localized.  
Up to the factor $1/6$ in the exponent, this is the best result of its kind that one can hope for. 
If the spectral gap $\delta$ between eigenvalues $\mu_{\ell,j}$ in adjacent $E_\ell$
is smaller
than $\bar V e^{-c\bar S}$ for some sufficiently large $c$, then the eigenfunction may be a linear combination with significant 
contributions from more than one $E_\ell$.

Finally, we describe the correspondence between actual eigenvalues and
localized eigenvalues $\mu_{\ell,j}$ up to exponential errors.  This, combined with 
Theorem \ref{thm:introspectralprojection}, gives the full picture of the correspondence between 
actual eigenfunctions and localized eigenfunctions $\vp_{\ell,j}$ up to exponential errors
for low eigenvalues.  Denote by $N_0(\lambda)$ the cumulative eigenvalue counting function for the union of the 
$\vp_{\ell,j}$ and by $N(\lambda)$ the counting function for the original operator $L$.

\begin{corollary} \label{cor:introeigenbound} (see Corollary~\ref{C4.3})  
Suppose that $\delta$, $\bar \mu$ and $\bar N$ are chosen so that 
\begin{equation}\label{eq:introeigenbound}
\lambda_1 \le \lambda_2 \le \cdots \le \lambda_{\bar N} \le \bar \mu -\delta, \quad  300 \bar N \left(\frac{\bar V}{\delta}\right)^3 e^{-\bar S/2} < 1. 
\end{equation}
Then 
\begin{equation}\label{eq:introeigencount}
N_0(\lambda -\delta) \le N(\lambda) \le N_0(\lambda + \delta), \quad \mbox{for} \  \lambda \le\bar \mu -\delta.
\end{equation}
\end{corollary}
The corollary follows readily from Theorem \ref{thm:introspectralprojection}.   
It says that the two eigenvalue counts coincide up to $\delta$ with 
$\delta \approx \bar V \,\bar N^{1/3} e^{-\bar S/6}$, parallel to \eqref{eq:spectralgap}.

The constants in our estimates depend  only on the spectral ratio $\delta/\bar V$, so 
 we can easily see the exponential take control as $T$ increases with the help of 
numerical experiments on $\RR/T \ZZ$.  
For each of  $T = 2^5, \, 2^6,\,  \dots \,, 2^{19}$, we carried
out 200 realizations of a potential $V$ with constant values on unit intervals,
chosen independently and uniformly distributed  between $0$ and $\bar V = 4$. 
We found that the gap $\lambda_2-\lambda_1$ is typically\footnote{
Lower bounds on spectral gaps are called Wegner type estimates.
In \cite{FS},
Fr\"ohlich and Spencer showed that  for large disorder,
the gap is bounded below by a multiple of $1/T^n$ with high probability
in the discrete Anderson model on $\ZZ^n$ with uniformly distributed $V$.
A similar conclusion holds with a larger power of $T$ in many cases in which $V$ has a singular  continuous distribution (see \cite{CKM}).  
}
greater than $1/T$.  
(This is nearly the same, by \eqref{eq:introeigenbound},    as the spectral
gap between the first two  {\em localized} eigenvalues $\mu_{\ell, j}$.)
The minimum separation $S$ between consecutive connected components 
of $E(\lambda_1 + 1/T)$ conforms very well to the power law $\mbox{median}(S) \sim .69 \,T^{.59}$.  
For $T =2^{15}$, the values $\delta_1 = 1/T$, $\bar V = 4$,  and the median $S_1 = .69\, T^{.59}$, we have
\[
300 \left(\frac{\bar V}{\delta_1}\right)^3 e^{-S_1/2} << 10^{-50}.
\]
Thus, \eqref{eq:introdiag} typically shows that the ground state $\psi_1$ is extremely close to a single spike.

Theorem \ref{thm:introspectralprojection} is motivated by work of 
Helffer and  Sj\"ostrand \cite{H,HS} and Simon \cite{S1,S2}  on resonance
  for smooth potentials $V$  in the semi-classical regime,
 $-h^2 \Delta +V$ as $h\to 0$,  but our potentials are much more irregular, our
eigenfunctions have a different shape, and our methods are different.
We use weak eigenfunction equations and derive error estimates in the dual space to the standard Sobolev space $W^{1,2}(M)$ (see \eqref{eq:error.r}).  By relying only on dual space estimates,
we can eliminate all dependence on smoothness, and express our results
explicitly in terms of the spectral gap ratio $\delta/\bar V$.    The dual estimates are just barely strong enough  to yield estimates for the spectral projection and eigenvalue distribution.

Because our exponential decay result is relative to distance
to all of $E(\lambda + \delta)$ rather than to a single well, it does not address directly the further decay we
see numerically as we pass through the second and third effective barrier, etc.  
Our subsequent estimates show that resonance is the main issue.   The natural conjecture is that the interaction of pairs of eigenfunctions depends primarily on
the effective distance between the wells or cluster to which they belong,
rather than the minimum distance $\bar S$ between all pairs of wells.
The proof can be expected to depend on multi-scale analysis and a more detailed
spectral gap hypothesis like the condition \eqref{eq:spectralgap} above, localized
to pairs or groupings of wells.  Showing that such a hypothesis is satisfied
with high probability should employ tools associated with so-called Wegner
estimates in the theory of Anderson localization.

\section{Main Assumptions and Preliminary Estimates}\label{sec:prelim}

Let $\Omega$ be  a bounded, connected, open subset of $\RR^n$ such
that at each boundary point the domain is locally equivalent
to a half space via a bi-Lipschitz mapping. 
(In Section~5, we will replace the ambient space $\RR^n$ with a compact
$C^1$ manifold $\wh M$.)

Set $M = \bar\Omega$, and let $m\in L^\infty(\Omega)$ be a real-valued density satisfying uniform upper and lower bounds 
\[
\frac1C \le m(x) \le C,
\]
for some positive constant $C$. Let $A= (a_{ij}(x))_{i,j=1}^n$ be a bounded measurable, real symmetric matrix-valued function, satisfying the uniform ellipticity condition
\begin{equation}\label{eq:ellipticity}
\frac1C |\xi|^{2}\leq  
\sum_{i,j=1}^{n}a_{ij}(x)\xi_{i}\xi_{j}
\leq C |\xi|^{2}, \quad x\in \Omega, \quad \xi \in \RR^n.
\end{equation}
for some $C<\infty$. We define the elliptic operator $L$ acting formally on real-valued functions $\vp$ by
\begin{equation} \label{eq:Lformaldef}
L\vp =-\frac1m\div (mA\nabla \vp) +V\vp =-\frac 1m 
\sum_{i,j=1}^n  
\frac{\partial}{\partial
x_{i}}\left(m \, a_{ij} \,\frac{\partial \vp}{\partial x_{j}}\right)+V \vp.
\end{equation}

The operator $L$ will always be used in the weak sense, defined as follows.
\begin{definition}\label{def:weakLphi=f}
A function $\vp\in W^{1,2}(\Omega)$ satisfies $L\vp=f$ weakly on $\Omega$ (respectively, on $M= \bar \Omega$) if 
\begin{equation}\label{eq:weakLphi=f} 
\int_\Omega [(A \nabla \vp )\cdot \nabla \eta +
 V  \vp  \eta ] \,m\, dx=\int_\Omega f \eta\, m \,dx 
\end{equation}
for every $\eta\in W_0^{1,2}(\Omega)$ (respectively, for every $\eta \in W^{1,2}(\Omega)$).
\end{definition}
Here the space $W^{1,2}(\Omega) = W^{1,2}(M)$ is the usual Sobolev space, namely the closure of $C^1(M)$ in the function space with square norm given by 
\[
\int_\Omega (|\nabla \vp|^2 + \vp^2) \, dx.
\]
The space $W_0^{1,2}(\Omega)$ is the closure in the same norm of the subspace $C_0^1(\Omega)$ of continuously differentiable functions that are compactly supported in $\Omega$. 
 
The weak equation on $M= \bar \Omega$ imposes, in addition to the interior condition, a weak form of the Neumann boundary condition on $\vp$. If there is sufficient smoothness to justify integration by parts, then the Neumann condition can be written 
\[
\nu(x) \cdot A(x)\nabla \vp(x) = 0, \quad x\in \partial \Omega,
\]
with $\nu $ the normal to $\partial \Omega$. In fact, in the case of Lipschitz boundaries, the Neumann condition is valid almost everywhere with respect to surface measure on $\partial \Omega$ for suitable right hand sides $f$. But, we will only need the weak form, not this strong version of the boundary condition.  (For now we confine ourselves to
Neumann boundary conditions; we will say a few words about Dirichlet and mixed boundary conditions later.)

We assume further that $V$ is non-degenerate in the sense that it is strictly positive on a subset of positive measure of $\Omega$.  By ellipticity of $A$ and the fact that $\Omega$ is a connected, bounded bi-Lipschitz domain, we have the coercivity inequality 
\[
\int_M [(A \nabla \vp) \cdot \nabla \vp
+ V \vp^2] \, m \, dx 
\ge c \int_M (|\nabla \vp|^2 + \vp^2)\, dx,
\]
for some $c>0$. In other words, the formal $L^2(M,m\,dx)$ inner product $\langle L\vp,\vp\rangle$ is comparable to the square of the $W^{1,2}(\Omega) = W^{1,2}(M)$ norm of $\vp$. By the Fr\'echet--Riesz theorem (identifying a Hilbert space with its dual), this implies that for every $f \in L^2(M, m\,dx)$, there is a unique solution $v\in W^{1,2}(M)$ to the weak equation $L v = f$ on $M$. The {\em landscape function} $u$ is defined as the solution to 
\[
Lu =1 \quad \mbox{weakly on } \ M.
\]
In other words, $u$ is the unique weak solution to the inhomogeneous Neumann problem with right hand side the constant $1$.

\begin{proposition} \label{prop:ubounds}
Let $V$ be nondegenerate and satisfy $0 \le V \le \bar V$ for some constant $\bar V$.  Then the landscape function $u\ge 1/\bar V$ on $M$. Moreover $u\in C^\alpha(M)$ for some $\alpha>0$.
\end{proposition}
\begin{proof}
Consider the weak solution to $Lv = f$ on $M$ for bounded measurable $f$.   H\"older regularity of $v$ at interior points of $M$ follows from a version of
the theorem of De Giorgi, Nash, and Moser (see Theorem 8.24, \cite{GT}).  Near each boundary point, one can define an ``even'' reflection of $v$ that satisfies a uniformly elliptic equation in a full neighborhood; hence $v$ is $C^\alpha$ up to the boundary for some $\alpha>0$. This reflection argument is presented in the last section in the more general context of manifolds (see Proposition~\ref{prop:boundaryHolder}). In particular, $u\in C^\alpha(M)$. 

Next, we prove a version of the maximum principle, namely that $v\ge 0$ provided $f\ge 0$. Since $v$ is continuous, the set $\Omega^- = \{x\in \Omega: v(x)<0\}$ is open. Since $v$ minimizes
\[
\int_\Omega \left( (A\nabla \vp)\cdot \nabla \vp + 
V \vp^2 - 2f\vp\right)
\, m\, dx
\]
among all $\vp\in W^{1,2}(M)$, we have
\begin{align*} 
\int_\Omega \left((A\nabla v)\cdot \nabla v + 
V v^2 - 2fv \right) &\, m\, dx 
\\& \hskip-3cm
\le \int_\Omega 
\left((A\nabla v_+)\cdot \nabla v_+ + 
V v_+^2 - 2fv_+\right)
\, m\, dx 
\end{align*}
for $v_+(x) = \max(v(x),0)$. Consequently, 
\[
\int_{\Omega^-} \left((A\nabla v)\cdot \nabla v + 
V v^2 - 2fv\right)
\, m\, dx \le 0.
\]
Because $V\ge0$ and $f\ge 0$, we have $Vv^2 - 2fv \ge 0$ on $\Omega^-$. Therefore,
\[
\int_{\Omega^-}(A\nabla v)\cdot \nabla v \ m\, dx \le 0.
\]
Since $A$ is coercive, $\nabla v = 0$ a.e. on $\Omega^-$, and $v$ is a strictly negative constant on each connected component of $\Omega^-$.  If any such component is a proper subset of $\Omega$, then the continuity of $v$ contradicts the fact that $v\ge 0$ on $\Omega\setminus \Omega^-$. On the other hand, if $\Omega^-  = \Omega$, then $v\equiv -a$, for some constant $a>0$. But in that case, $Lv = -aV$, which cannot equal $f\ge 0$. Thus, the only possibility is that $\Omega^-$ is empty.

Finally, to conclude proof of the proposition, consider $u$, the  weak solution to $Lu=1$ on $M$. Then
\[
v = u - \frac1{\bar V} \quad \mbox{solves} \quad L v = 1 - \frac{V}{\bar V} \ge 0.
\]
Therefore, by the maximum principle, $v\ge0$, and $u \ge 1/\bar V$. 
\end{proof}

By the bi-Lipschitz assumption on $\Omega$ and the Rellich-Kondrachov lemma, the inclusion mapping $W^{1,2}(M)\hookrightarrow L^2(M)$ is compact. Thus, by the spectral theorem for compact operators, there is a complete orthonormal system of eigenfunctions to the Neumann problem for $L$, that is, an orthonormal basis $\psi_j$ of $L^2(M)$ such that $\psi_j\in W^{1,2}(M)$, and 
\[
L \psi_j = \lambda_j \psi_j \quad \mbox{weakly on $M$}.
\]
The non-degeneracy of $V$ implies that the eigenvalues $\lambda_j$ are strictly positive,

We will compare these eigenfunctions to localized eigenfunctions of Dirichlet or mixed boundary value problems. Let $K$ be a compact subset of $M$.
Let $U$ be a connected component
of $M\setminus K$.  We say that $L\vp =f$ weakly on $U$ if equation \eqref{eq:weakLphi=f} holds for all test functions $\eta\in C^1(M)$ such that the support of $\eta$ is contained in $U$.  
We will denote the closure of this set of test functions in the usual $W^{1,2}(\RR^n)$ 
norm by  $W^{1,2}_0(U)$. Formally, solutions to $L\vp=f$ on $U$ 
satisfy mixed boundary conditions
\[
\vp (x) = 0, \quad x\in K\cap \partial U; 
\quad 
\nu(x) \cdot A(x)\nabla \vp(x) = 0, \quad x\in (\partial \Omega) \cap \partial U.
\]
In the special case $K\supset \partial \Omega$, the problem is no longer mixed because we only have Dirichlet boundary conditions.  We won't need the Neumann
boundary equations in strong form, only the weak, integrated form.  
On the other hand, we will use continuity of the solutions up to the boundary.  
In fact, we will obtain $C^\alpha$ regularity.

To ensure the H\"older regularity of solutions we make an additional assumption on 
the compact set $K\subset M$. 
We will say that $K$ satisfies the {\em bi-Lipschitz cone condition}
if there are $r>0$ and $\eps>0$ such that at every point $x_0\in \partial K$ 
there is a mapping $F: B_r (x_0) \to \RR^n$ with $F(x_0) = 0$,
bi-Lipschitz bounds $\eps |x-y| \le |F(x)-F(y)| \le (1/\eps) |x-y|$, and 
such that 
\[
F(K) \supset \{x= (x_1,x')\in \RR\times \RR^{n-1}:  |x'|  < \eps x_1 < \eps^2\}.
\]
The constants in our main theorems do not depend on $r$, $\eps$
or the bi-Lipschitz constants of $\Omega$ 
because  continuity of the solutions is only used in a qualitative way.

\begin{proposition} \label{prop:phi.j.Holder}
Suppose that $V$ is nondegenerate, $K$ is a compact subset of $M$
satisfying the bi-Lipschitz cone condition. 
Let $U$ be a connected component of $M\setminus K$.
Then there is an orthonormal basis $\vp_j$ of $L^2(U,  m\,dx)$ of eigenfunctions solving $L\vp_j = \mu_j \vp_j$ weakly on $U$, $\mu_j>0$. After extending the functions $\vp_j$
from 
$U$ to the rest of $M$ by $\vp_j= 0 $ on $M \setminus U$, 
they satisfy  $\vp_j\in C^\alpha(M)\cap W^{1,2}(M)$ for some $\alpha>0$.
\end{proposition}

The proof of the existence of the complete orthonormal basis of eigenfunctions is the same as in the case of $K=\emptyset$, that is, the case of $\psi_j$ above. 
 See Proposition~\ref{prop:boundaryHolder} for the proof $C^\alpha$ regularity.
(At interior points the proof is similar to the case of $Lv=f$ above.
 The boundary regularity is proved by reducing to a Dirichlet problem using an even reflection.)

\section{Agmon estimates} \label{sec:agmon}

We will frequently write
\[
\nabla_A = A^{1/2} \nabla
\]
in which $A^{1/2}=A^{1/2}(x)$ is the positive definite square root of the matrix $A(x)$ and $\nabla$ is a column vector. Thus, we have 
\[
 \nabla_A \vp \cdot \nabla_A \eta 
= ( A \nabla \vp ) \cdot( \nabla \eta);
\quad |\nabla_A\vp|^2 = (A \nabla \vp)\cdot \nabla \vp. 
\]

\begin{lemma} \label{lem:identity} Assume that $f$ and $u$ belong to $W^{1,2}(M)$,  that $V$, $f$, and $1/u$ belong to $L^\infty(M)$, and that $u$ satisfies $Lu = 1$ weakly on $M$. Then 
\[
\int_M( |\nabla_A f|^2 + Vf^2)\, m \, dx  
= \int_M \left(u^2 \left| \nabla_A (f/u)\right|^2 +\frac1u f^2  \right) \, m\, dx.
\]
\end{lemma}
\begin{proof}
The function $f^2/u$ belongs to $W^{1,2}(M)$, so we may use it as a test function in the weak form of $Lu = 1$ to obtain
\begin{equation*}
\int_M [(\nabla_A u \cdot \nabla_A(f^2/u) ) + V u(f^2/u) ] \, m \, dx  
= \int_M (f^2/u)\, m \, dx  .
\end{equation*}
Substituting the identity $\nabla_A u \cdot \nabla_A (f^2/u)=|\nabla_A f|^2 - u^2 |\nabla_A (f/u)|^2$ (from the product rule), this becomes
\[
\int_M( |\nabla_A f|^2  - u^2 |\nabla_A (f/u)|^2 + Vf^2)\, m \, dx = \int_M (f^2/u)\, m \, dx,
\]
which, after moving a term from the left to the right, is the desired result.
\end{proof}

Given the importance of Lemma~\ref{lem:identity} to this paper, we wish to elaborate on it, recapitulating
the introduction with more details. Recall that 
\[
Lf = -\frac1m \div(mA\nabla f) + Vf
\]
in the weak sense. Define the operator
$\tilde L$ by 
\[
\tilde L g := \frac1u L(gu).
\]
In other words, $\tilde L$ is the conjugation of $L$ by the operator multiplication by $u$. If the functions $m$ and $A$ are differentiable, then one can use equation $Lu=1$ to compute that 
\[
\tilde L g = -\frac1{mu^2} \div (mu^2 A \nabla g) + \frac1u g \, .
\]
Note that the operator $\tilde L$ is of the same form as $L$ but with a different density and potential. The key point is that the potential $V$ in $L$ has been replaced by the potential $1/u$ in $\tilde L$. Mechanisms of this type are familiar in the theory of second order differential equations. Conjugation of operators of the form $-\Delta + V$ using an auxiliary solution is a standard device leading to the generalized maximum principle  (see Theorem 10, page 73 \cite{PW}). A similar device appears even earlier in work of Jacobi on conjugate points and work of Sturm on oscillation of eigenfunctions. In all of these cases, the multipliers are eigenfunctions or closely related supersolutions rather than solutions to the equation $Lu=1$. 

Consider the space $L^2(M,m\,dx)$ with inner product $\langle \,\cdot\, , \, \cdot\, \rangle$. The operators $L$ and $u^2 \tilde L$ are self adjoint in this inner product. Using the formula for $\tilde L$ above, one could derive the lower bound $\langle Lf,f\rangle \ge \langle (1/u)f, f\rangle$ formally by substituting $f = gu$:
\[
\langle Lf, f\rangle 
= \langle u^2 \tilde L g, g\rangle 
\ge \langle u^2 (1/u)g, g\rangle 
= \langle (1/u)f,f\rangle.
\]
Lemma~\ref{lem:identity} implies that the identity $\langle Lf, f\rangle = \langle u^2 \tilde L g, g\rangle$ is valid in weak form. Indeed, it says that 
\[
\langle Lf, f\rangle = \int_M( |\nabla_A f|^2 + Vf^2)\, m \, dx
= \int_M \left[
u^2 |\nabla_A (f/u)|^2 + \frac1u f^2  \right] \, m\, dx,
\]
and so, since $g=f/u$, 
\[
 \langle Lf, f\rangle 
 =  \int_M u^2 \left[|\nabla_A g|^2 + \frac1u g^2\right] \, m \, dx
 =  \langle u^2 \tilde L g, g\rangle.
 \]

Although conjugation and the calculation of $\tilde L$ leads to
our identity, the weak form has considerable advantages. It is easier to check the weak formula than the differential formula for $\tilde L$ because it only involves first derivatives.  Moreover,
because
we only differentiated once and didn't integrate by parts,
our proof of Lemma~\ref{lem:identity} 
was not only shorter but also more general in that it applied to bounded measurable $m$ and $A$.  

We will now derive estimates of Agmon type from Lemma~\ref{lem:identity}.

\begin{lemma}\label{lem:eigenidentity}
Suppose $\vp$ belongs to $W^{1,2}(M)\cap C(M)$, $\vp = 0$ on a compact subset $K$ of $M$ and $L\vp = \mu \vp$ weakly on $M\setminus K$. Let $u$ be as in Lemma~\ref{lem:identity} and let $g$ be a Lipschitz function on $M$.  Then
\begin{equation}\label{eq:eigenidentity}
\int_M \left[u^2 |\nabla_A (g\vp/u)|^2 
+ \left(\frac1u - \mu\right) (g\vp)^2\right] \, m \, dx  
= 
\int_M| 
\nabla_A 
g|^2 \vp^2 \, m \, dx  \, .
\end{equation}
Furthermore, setting $g = \chi e^h$ with $h$ and $\chi$ Lipschitz functions on $M$, we have
\begin{multline} \label{eq:agmonidentity}
\int_M u^2 
\left| \nabla_A \left(\frac{\chi e^h \vp}{u}\right)\right|^2 \, m \, dx  
 \ +  
\int_M  \left(\frac1u-\mu - |\nabla_A h|^2\right) (\chi e^h \vp)^2 \, m \, dx   \\
 = \int_M 
\left(|\chi \nabla_A h + \nabla_A \chi|^2 
- |\chi \nabla_A h|^2\right)(e^h\vp)^2\, m \, dx.
\end{multline}
\end{lemma}

\begin{proof} 
Since $g^2 \vp \in W^{1,2}(M)$ and $g^2\vp= 0$ on $K$, it can be used as a test function for the equation $L\vp = \mu \vp$, yielding
\begin{equation}\label{eq:Lphi}
\int_M(V-\mu)g^2 \vp^2 \, m \, dx
= -\int_M \nabla_A \vp \cdot \nabla_A (g^2\vp) 
\, m\, dx \, .
\end{equation}
Substituting $f= g\vp$ in Lemma~\ref{lem:identity}, gives 
\begin{multline*}
\int_M \big[|\nabla_A (g\vp)|^2 + (V-\mu)g^2 \vp^2 \big] \, m \, dx  
\\= 
\int_M \Big[u^2 | \nabla_A (g\vp/u)|^2 
+\Big(\frac1u - \mu\Big) g^2 \vp^2 \Big] \, m \, dx  .
\end{multline*}
On the other hand, \eqref{eq:Lphi} implies that
\begin{multline*}
\int_M [|\nabla_A (g\vp)|^2 + (V-\mu)g^2 \vp^2] \, m \, dx  \\
= \int_M [|\nabla_A (g\vp)|^2 -  \nabla_A \vp \cdot \nabla_A (g^2\vp)] \, m \, dx  
= \int_M \vp^2 |\nabla_A g|^2 \, m \, dx.  
\end{multline*}
This proves \eqref{eq:eigenidentity}. The second formula,\eqref{eq:agmonidentity}, follows from the first, by setting $g = \chi e^h$, and using the formula
\[
|\nabla_A g|^2 
= |\nabla_A (\chi e^h)|^2 = (\chi e^h)^2 |\nabla_Ah|^2 
+ (|\chi \nabla_A h + \nabla_A\chi|^2 - |\chi \nabla_A h|^2) e^{2h}.
\]
\end{proof}

Let $w$ be a nonnegative, continuous function on $M$. Assume the elliptic matrix $A$ is continuous on $M$. Denote the entries of $B = A^{-1}$ by $b_{ij}(x)$.  We define the distance $\rho(x,y)$ on $M$ for the degenerate Riemannian metric $\displaystyle ds^2 = w(x) \sum b_{ij}dx_idx_j$ by
\begin{equation}\label{eq6.3}
\rho(x,y)=\inf_\gamma \int_0^1 \Bigl(w(\gamma(t))\sum_{i,j=1}^n
b_{ij} (\gamma(t)) \dot\gamma_i(t)\dot \gamma_j(t) \Bigr)^{1/2}\,dt,
\end{equation}
where the infimum is taken over all absolutely continuous paths $\gamma:[0,1]\to M$ such that $\gamma(0)=x$ and $\gamma(1)= y$.  (Note that the distance between points in a connected component of the set $\{w=0\}$ is zero.)

With these notations, we have the following lemma.
\begin{lemma}[\hbox{[\citenum{AgmonBook}, Theorem~4, p.~18]}]\label{lem:Agmondist}
If $h$ is  real-valued and $|h(x)-h(y)| \leq \rho(x,y)$ for all $x,y \in M$, then $h$ is a Lipschitz function, and 
\[
|\nabla_A   h(x)|^2\leq w(x) \quad \mbox{for all } \ x\in M.
\]
In particular, this holds when 
\[
h(x) = \inf_{y\in E} \rho(x,y),
\]
for any nonempty set $E\subset M$.
\end{lemma}
The lemma is stated in \cite{AgmonBook} for $w$ strictly positive. Considering the case $w(x) + \epsilon$ and taking the limit as $\epsilon \searrow 0$ gives the result for non-negative $w$.

Recall that $V$ is a measurable function on $M$ such that $0 \le V(x) \le \bar V$, and $V$ is nonzero on a set of positive measure and $u$ is the unique weak solution to $Lu = 1$ on $M$, the landscape function.

Fix $\mu\ge0$, and set
\[
w_\mu(x) = \left(\frac1{u(x)} - \mu \right)_+ 
= \max
\left(\frac1{u(x)} - \mu,\, 0\right).
\]
With our additional assumption that the elliptic matrix $A$ has continuous coefficients on $M$, we can define $\rho_\mu(x,y)$ as the Agmon distance associated to the weight $w_\mu(x)$. For any $E\subset M$, denote
\[
\rho_\mu(x, E) = \inf_{y\in E} \rho_\mu(x,y).
\]
\begin{theorem} \label{thm:expdecay1} 
Let
$0 \le \mu \le \nu \le \bar V$ be constants. With $u$ the landscape function as above, denote
\[
E({\nu}) = \{x\in M: \frac1{u(x)} \le \nu \}.
\]
Let $K$ be a compact subset of $M$. Denote 
\[
h(x) = \rho_\mu(x,E({\nu}) \setminus K), \quad x\in M,
\]
and 
\[
\chi(x) = 
\begin{cases} 
h(x), \ & h(x) < 1, \\
1, \ & h(x) \ge 1. 
\end{cases}
\]
Suppose $\vp$ belongs to $W^{1,2}(M)\cap C(M)$, $\vp = 0$ on $K$, and $L\vp = \mu \vp$ weakly on $M\setminus K$. Then for $0 < \alpha < 1$, 
\begin{multline}\label{eq:expdecay1}
\int_{M} u^2 
 \left| \nabla_A \left( \frac{\chi e^{\alpha h}\vp}u \right) \right|^2 \, m \, dx  
 + (1-\alpha^2)  \int_{M}
\left(\frac1u-\mu \right)_+ \left(\chi e^{\alpha h} \vp\right)^2 \, m \, dx  
\\
 \le (1+2\alpha)e^{2\alpha}(\bar V - \mu) \int_{\{0< h< 1\}} \vp^2 \, m \, dx.
\end{multline}
Furthermore, if $\nu = \mu + \delta$, $\delta >0$, we have 
\begin{equation}\label{eq:expdecay2}
\int_{h\ge 1} e^{2\alpha h}
\left(|\nabla_A\vp|^2 + \bar V \vp^2\right)\, m \, dx  
\le \left(450 + \frac{130\bar V}{(1-\alpha)\delta} \right) 
\bar V \int_M \vp^2 \, m \, dx. 
\end{equation}
\end{theorem}
\begin{proof} Using \eqref{eq:agmonidentity} with  $\alpha h$ in place
of $h$, the first term on the left side is the same as in \eqref{eq:expdecay1}. Since $\chi = 0$ on $E_\mu \setminus K$ and $\vp = 0$ on $K$, we have  $\chi \vp = 0$ on $E_\mu$.  Moreover, by Lemma~\ref{lem:Agmondist} $|\nabla_A h|^2 \le w_\mu(x)$.  Thus,
\begin{align*} 
\int_M \left(\frac1u -\mu - \alpha^2 |\nabla_A h|^2\right) (\chi e^{\alpha h}\vp)^2
&\, m \, dx  
\\& \hskip-2cm
= \int_{M\setminus E_\mu}  \left(\frac1u -\mu - \alpha^2 |\nabla_A h|^2\right) 
(\chi e^{\alpha h}\vp)^2 \, m \, dx  
\\& \hskip-2cm
\ge (1-\alpha^2) \int_{M\setminus E_\mu}  \left(\frac1u -\mu\right)_+ 
(\chi e^{\alpha h}\vp)^2 \, m \, dx  
\\&\hskip-2cm 
= (1-\alpha^2) \int_{M}  \left(\frac1u -\mu\right)_+ 
(\chi e^{\alpha h}\vp)^2 \, m \, dx  \, .
\end{align*}

The right side integrand of \eqref{eq:agmonidentity} is zero almost everywhere on the set $\nabla_A \chi =0$, so we may restrict the integral to the set $\{0 < h < 1\}$. There we have $\chi\equiv h$, so
\[
|\chi \alpha \nabla_A h + \nabla_A \chi|^2 - |\chi \alpha \nabla_A h|^2
= [(\chi \alpha + 1)^2 - \chi^2 \alpha^2] |\nabla_A h|^2 
\le (2\alpha +1)|\nabla_Ah|^2 \, .
\]
Finally, $|\nabla_A h|^2  \le w_\mu(x) \le \bar V - \mu$, by Lemma~\ref{lem:Agmondist} and Proposition~\ref{prop:ubounds}. This concludes the proof of \eqref{eq:expdecay1}.

It remains to prove \eqref{eq:expdecay2}. For convenience, normalize $\vp$ so that its $L^2(M,m\,dx)$ norm is $1$:
\[
\|\vp\|^2 := \int_M \vp^2 \, m \, dx  = 1.
\]
Let $f = \chi e^{\alpha h} \vp$. Since $f=0$ on $E(\nu)$, $(1/u - \mu) \ge \delta$ on $M\setminus E(\nu)$, and $\mu \ge 0$, \eqref{eq:expdecay1} implies
\begin{equation}\label{eq:decayf}
\int_M u^2 \left| \nabla_A (f/u)\right|^2\, m \, dx  
 + (1-\alpha^2)\delta\int_M f^2 \, m \, dx  
\le
(1+2\alpha)e^{2\alpha} \bar V \, .
\end{equation}
Since $\nabla f$ and $\nabla u$ belong to $L^2(M)$, and $1/u$ and $f$ belong to $L^\infty(M)$, $f^2/u$ is a permissible test function. Thus, using $Lu = 1$, $1/u(x)\le \bar V$, $V(x) \ge 0$, and \eqref{eq:decayf}, we have 
\begin{align}\label{3.13a}
\int_M \nabla_A u &\cdot \nabla_A (f^2/u) \, m \, dx  
= \int_M (1-Vu)(f^2/u) \, m \, dx  
\nonumber \\ 
&\le \bar V \int_M f^2  \, m \, dx \le
\frac{(1+2\alpha)e^{2\alpha}}{(1-\alpha^2)\delta } \bar V^2
\le \frac{3e^2}{2(1-\alpha)\delta} \bar V^2 \, .
\end{align}

Next, 
\begin{align*} 
\int_M & |\nabla_A u|^2 ( f  /u)^2  \, m \, dx = -\int_M 2(f/u)(\nabla_A u) \cdot (u\nabla_A(f/u))\, m \, dx  
\\
& \hspace{180pt} + \int_M \nabla_A u \cdot \nabla_A(f^2/u)\, m \, dx   \\
&\hspace{10pt} \le 
\int_M \left[\frac12 (f/u)^2 |\nabla_Au|^2 + 2 u^2 |\nabla_A(f/u)|^2 
+ \nabla_Au \cdot \nabla_A(f^2/u)\right] \, m \, dx.
\end{align*}

Hence, after subtracting the term with factor $1/2$ and multiplying by $2$,\begin{align*}
\int_M |\nabla_A u|^2  (f /u)^2  \, m \, dx  
& \le
\int_M [4u^2 |\nabla_A(f/u)|^2 + 2\nabla_Au  \cdot \nabla_A(f^2/u) ] 
\, m \, dx  \\
& \le 4(1+2\alpha) e^{2\alpha} \bar V + 3e^2 \frac{\bar V^2}{(1-\alpha) \delta} \\
&\le 12 e^{2} \bar V + 3e^2 \frac{\bar V^2}{(1-\alpha) \delta} \, .
\end{align*}
It follows that
\begin{equation} \label{3.14a}
\begin{aligned}
\int_M |\nabla_A f|^2 \, m \, dx  
& =  \int_M 
|u \nabla_A (f/u) 
+ (f/u)\nabla_A u|^2 \, m \, dx  
\\
& \le 
2\int_M u^2 |\nabla_A (f/u)|^2\, m \, dx  
 + 2\int_M |\nabla_A u|^2 (f/u)^2 \, m \, dx  
 \\
& \le  2(1+2\alpha) e^{2\alpha} \bar V + 2\left[12 e^{2} \bar V
 + 3e^2 \frac{\bar V^2}{(1-\alpha) \delta}\right]\, m \, dx   
\\
& \le  30 e^{2} \bar V + 6 e^2 \frac{\bar V^2}{(1-\alpha) \delta} \, .
\end{aligned}
\end{equation}
Finally, since $e^{\alpha h} \vp = f$ on $\{h \ge 1\}$, and $|\nabla_A h|^2 \le \bar V$, we have (by \eqref{3.13a} and \eqref{3.14a} in particular)
\begin{equation} \label{3.15a}
 \begin{aligned}
  \int_{\{ h\ge 1\}} 
e^{2\alpha h} &|\nabla_A \vp|^2 \, m \, dx  
= \int_{\{ h\ge 1\}} 
|\nabla_A (e^{\alpha h}\vp) - \alpha (\nabla_A h) e^{\alpha h} \vp|^2 
\, m \, dx \\
&\le
2\int_{\{ h\ge 1\}} |\nabla_A (e^{\alpha h}\vp)|^2 \, m \, dx  
+ 2 \int_{\{h\ge 1\}}
\alpha^2 |\nabla_A h|^2 (e^{\alpha h} \vp)^2 \, m \, dx \\
&\le 2\int_{\{ h\ge 1\}} |\nabla_A f|^2\, m \, dx   
+ 2 \bar V \int_{\{h\ge 1\}}  f^2 \, m \, dx   \\ 
&\le 60 e^2 \bar V + 12 e^2 \frac{\bar V^2}{(1-\alpha) \delta} 
+ 3e^2 \frac{\bar V^2}{(1-\alpha) \delta} \, . 
 \end{aligned}
\end{equation}

Thus, by \eqref{3.13a} again,
\begin{align*}
\int_{\{ h\ge 1\}} 
e^{2\alpha h} (|\nabla_A \vp|^2 + \bar V \vp^2) \, m \, dx  
&\le 60 e^2 \bar V + 15 e^2 \frac{\bar V^2}{(1-\alpha) \delta} 
+ \frac32 e^2 \frac{\bar V^2}{(1-\alpha) \delta}
\\
&\le \left( 450 + \frac{130\bar V}{(1-\alpha)\delta}\right) \bar V \, .
\end{align*}
\end{proof}

Theorem~\ref{thm:expdecay1} displays the dependence of the constant as $\alpha \to 1$. We state next a variant for $\alpha=1/2$ in the form we will use below.

\begin{corollary} \label{cor:expdecay} 
Let $0< \mu \le \bar \mu$ and $0 < \delta \le \bar V/10$ be constants.  Suppose that $ \bar \mu + \delta \le \bar V$. Let $K$ be a compact subset of $M$, and set
\[
h_K(x) = \bar \rho (x,E({\bar \mu + \delta }) \setminus K), \quad x\in M,
\]
with $\bar \rho = \rho_{\bar \mu}$ the Agmon metric associated to the weight $\bar w(x) = (1/u(x) - \bar \mu)_+$. Suppose $\vp$ belongs to $W^{1,2}(M)\cap C(M)$, $\vp = 0$ on $K$, and $L\vp = \mu \vp$ weakly on $M\setminus K$.   Then 
\begin{equation}\label{eq:expdecay.phi}
\int_{h_K\ge 1} e^{h_K}
\left(|\nabla_A\vp|^2 + \bar V \vp^2\right) \, m \, dx  
\le 18e\left(\frac{\bar V}{\delta}\right)
\bar V \int_M \vp^2\, m \, dx. 
\end{equation}
\end{corollary}

In particular, in the case $K=\emptyset$, the corollary says that for eigenfunctions $\psi$ satisfying $L\psi = \lambda\psi$ weakly  on all of $M$ for which $\lambda \le \bar \mu$, we have 
\begin{equation}\label{eq:expdecay.psi}
\int_{h\ge 1} e^{h}
\left(|\nabla_A\psi|^2 + \bar V \psi^2\right)\, m \, dx  
\le 18e\left(\frac{\bar V}{\delta}\right)
\bar V \int_M \psi^2 \, m \, dx. 
\end{equation}
with 
\[
h(x) = \bar \rho (x,E(\bar \mu + \delta)), \quad x\in M.
\]

\begin{proof}  Corollary~\ref{cor:expdecay} is not, strictly speaking, a corollary of Theorem~\ref{thm:expdecay1}, but rather the specialization of the inequalities in the proof to the case $\alpha=1/2$. Note also the theorem is proved for $\mu = \bar \mu$, but the corollary is also valid for any larger value of $\bar \mu$. This because increasing $\bar \mu$ gives rise to a weaker conclusion: it decreases $h_K$.

Rather than repeat the proof, we indicate briefly the arithmetic that ensues from setting $\alpha = 1/2$ in the proof of Theorem~\ref{thm:expdecay1}. With $f = \chi e^{h_K/2} \vp$ and the normalization $\|\vp\|=1$, we have
\[
\int_{\{h_K\ge 1\}} e^{h_K} \bar V \vp^2\, m \, dx   \le 
\bar V \int_M f^2 \, m \, dx  
\le \frac{8e}{3} \frac{\bar V^2}{\delta},
\]
as in the second line of \eqref{3.13a},
\[
\int_M |\nabla_A f|^2 \, m \, dx  
\le \left(20 + \frac{3\bar V}{\delta}\right)  e\bar V\,  ,
\]
by the proof of \eqref{3.14a}, and (as for \eqref{3.15a})
\[
\int_{\{h_K\ge 1\}}    e^{h_K} |\nabla_A \vp|^2\, m \, dx  
 \le \left(40 + \frac{34\bar V}{3\delta}\right) e\bar V  \, .
\]
Therefore, again with the normalization $\|\vp \|=1$, 
\[
\int_{\{ h_K\ge 1\}} 
e^{h_K} (|\nabla_A \vp|^2 + \bar V \vp^2) \, m \, dx  
\le \left(40 + 14 \frac{\bar V}{\delta} \right) 
e \bar V \le 18e \left(\frac{\bar V}{\delta} \right) \bar V,
\]
where we have used $\delta \le \bar V /10$ to obtain the last inequality.
\end{proof}

\section{Localized approximate eigenfunctions}\label{sec:diagonalization}

We have already proved a theorem about exponential decay of the eigenfunctions $\psi$. We will now show, roughly speaking, that if the landscape function predicts localization, then an eigenfunction with eigenvalue $\lambda$ is localized in the components of $\{1/u \le \lambda\}$ where an appropriate localized problem has an eigenvalue in the range $\lambda \pm \delta$. 

Let $\bar\mu$ and $\delta$ be as in Corollary~\ref{cor:expdecay}. Consider any finite decomposition of the sublevel set $E(\bar \mu + \delta)$ into subsets:
\[
E(\bar \mu + \delta) = \{x\in M: \frac1{u(x)} \le \bar \mu  + \delta\} = \bigcup_{\ell = 1}^R E_\ell \, .
\]
We regard the sets $E_\ell$ as potential wells. It is easiest to visualize $E_\ell$ as the (closed) connected components of $E(\bar \mu + \delta)$. In practice, such connected wells often yield the optimal result. But there is no requirement that $E_\ell$ be connected. Rather each $E_\ell$ should be chosen to consist of a collection of ``nearby" wells. It is occasionally useful to merge nearby wells because what is important is to choose the sets $E_\ell$ so as to have a large separation between them, where the separation $\bar S$ is defined by 
\[
\bar S  =  \inf\ \{\bar \rho(x,y): x\in E_\ell, \ y\in E_{\ell'}, \ \ell \neq \ell'\}, 
\]
i.e., the smallest effective distance between wells.  Here, as before, $\bar \rho = \rho_{\bar \mu}$ denotes the Agmon metric associated to the weight $\bar w(x) = (1/u(x) - \bar \mu)_+$. Whether or not a decomposition into small, well-separated wells exists depends on the level set structure of $1/u(x)$ and the size of $\bar \mu + \delta$.

Let $S_1< \bar S$ (as near to $\bar S$ as we like). 
We claim that there is a compact set $K_\ell\subset M = \bar \Omega$ 
satisfying the hypothesis of Proposition \ref{prop:phi.j.Holder}
and such that 
\begin{equation}\label{4.1a}
\{x\in M: \bar \rho(x,E_\ell) \ge \bar S/2\} 
\ \subset \,K_\ell \ \subset \{x\in M: \bar \rho(x,E_\ell) >  S_1/2 \}.
\end{equation}
In fact, as we will show in Lemma \ref{lem:cubecover}, 
for any compact $K\subset M$ and any neighborhood $U\supset K$
(that is, $U$ is relatively open in $M$) 
there is an intermediate set $K\subset K' \subset U$
such that $K'$ satisfies the bi-Lipschitz cone condition.

Define $\Omega_\ell$ as the connected component of $M\setminus K_\ell$ 
containing $E_\ell$. 
Because the sets $E_\ell$ are at least distance $\bar S$ apart, 
the sets $\Omega_\ell$ are disjoint.

Denote by $W_0^{1,2}(\Omega_\ell)$ the closure in $W^{1,2}(M)$ norm of the space of smooth functions that are compactly supported on $\Omega_\ell$. Note that these functions can be extended by zero on $M \setminus \Omega_\ell$ and regarded as belonging to $W^{1,2}(M)$. But the notation is slightly misleading, because $\Omega_\ell$ is not necessarily open, and may contain parts of $\d M$ that do not lie in $K_\ell$. On those parts, functions of $W_0^{1,2}(\Omega_\ell)$ do not need to vanish. In other words, our definition of $W_0^{1,2}(\Omega_\ell)$ includes a Dirichlet condition on $K_\ell\cap \partial \Omega_\ell$ only.

The operator $L$ is self-adjoint with our mixture of Dirichlet and Neumann conditions, and for each $\ell$ there a complete system of orthonormal eigenfunctions $\vp_{\ell,j}\in W^{1,2}_0(\Omega_\ell) $ satisfying
\[
\int_M [\nabla_A \vp_{\ell,j}  \cdot \nabla \zeta + V\vp_{\ell,j}  \zeta]\, m \, dx
  = \mu_{\ell,j} \int_M \vp_{\ell,j} \zeta \, m \, dx  
\]
for all test functions $\zeta$ in $W^{1,2}_0(\Omega_\ell) $. We have Dirichlet conditions on $K_\ell \cap \partial \Omega_\ell$. If $\partial \Omega_\ell \cap \partial M $ is non-empty, then on that portion of the boundary, the weak equation is interpreted as a Neumann condition. But we will never have to use normal derivatives, only the weak equation. The purpose of inserting the somewhat nicer domain $\Omega_\ell$ is so that the eigenfunctions $\vp_{\ell,j}$ are continuous (in fact H\"older continuous) on $M$.  We do this so that the integrals in the lemmas above are well defined.  None of our inequalities with exponential weights depend on the Lipschitz constant of $\Omega_\ell$, just as they don't depend on the ellipticity constant or modulus of continuity of $A$. 

Let $\psi_j$ denote the complete system of orthonormal eigenfunctions of $L$ on $M$ with eigenvalues $\lambda_j$. Let $\Psi_{(a,b)}$ denote the orthogonal projection in $L^2(M, m\,dx)$ onto the span of eigenvectors $\psi_j$ with eigenvalue $\lambda_j \in (a,b)$. Let $\Phi_{(a,b)}$ be the orthogonal projection onto the span of the eigenvectors $\vp_{\ell,j}$ with eigenvalue $\mu_{\ell,j} \in (a,b)$.  Thus the range of $\Phi_{(0,\infty)}$ is 
the subspace of $L^2(M,m\, dx)$ of functions supported on $\displaystyle \cup_\ell \Omega_\ell$.
\begin{theorem} \label{thm:spectralprojection} 
Let $0 < \delta \le \bar V/10$. If $\vp$ is one of the $\vp_{\ell,j}$ with eigenvalue $\mu = \mu_{\ell,j}$ and $\mu \le \bar \mu - \delta$, and $\bar S$ is the effective distance separating wells, defined above, then 
\[
\| \vp - \Psi_{(\mu-\delta, \mu + \delta)}\vp \|^2 \le 300 \left(\frac{\bar V}{\delta}\right)^3 e^{-\bar S/2} \|\vp||^2,
\]
where here and below, $\| \cdot \|$ denotes the norm in $L^2(M, m\,dx)$. 
If $\psi$ is one of the $\psi_{j}$ with eigenvalue $\lambda  = \lambda_j\le \bar \mu - \delta$, then 
\[
\| \psi - \Phi_{(\lambda-\delta, \lambda+ \delta)}\psi \|^2
\le 300 \left(\frac{\bar V}{\delta}\right)^3 e^{-\bar S/2} \|\psi\|^2. 
\]
\end{theorem}

Note that this theorem only has content if $\bar S$ is sufficiently large  that
\[
\frac {\bar V}{\delta}  \ll e^{\bar S/6}.
\]
The separation $\bar S$ increases as $\bar\mu$ decreases.  Recall, also, that we have the flexibility to choose the sets 
$E_\ell$ so as to merge nearby wells that are not sufficiently separated. It turns out that the partition into well-separated wells does occur with high probability for many classes of random potentials $V$. 

\begin{proof}   Here and in the remainder of the paper all eigenfunctions are normalized to have $L^2(m\,dx)$ norm 1. 
Consider $\vp$ such that $L\vp = \mu \vp$ in the weak sense on $\Omega_\ell$. Let 
\[
\eta(x) = f(\bar \rho(x,E_\ell))
\]
be defined by 
\[
f(t) = \begin{cases} 
 1, & \quad t\le \frac{S_1}2 -1 \\
 \frac{S_1}2 - t, & \quad  \frac{S_1}2 - 1 \le t \le \frac{S_1}2 \\
 0, & \quad \frac{S_1}2 \le t
 \end{cases}
 \]
 Let $r$ be the distribution satisfying the equation 
\[
L(\eta \vp) = \mu \eta \vp + r
\]
in the weak sense on $M$. In other words, $r$ is defined by
\[
r(\zeta) := 
\int_M [\nabla_A(\eta\vp)\cdot \nabla_A \zeta + (V-\mu) \eta\vp \zeta] \, m \, dx  
\]
for all $\zeta$ smooth functions on $M$. Since $\eta \zeta$ is a suitable test function for $L\vp =\mu \vp$ in $\Omega_\ell$, we have
\[
\int_M [\nabla_A(\vp)\cdot \nabla_A (\eta \zeta) + (V-\mu) \eta\vp \zeta ] 
\, m \, dx  = 0\, .
\]
Subtracting this formula from the previous one for $r$, we find that
\[
r(\zeta) 
= \int_M [\vp \nabla_A \eta \cdot \nabla_A \zeta - \zeta \nabla_A\vp \cdot \nabla_A \eta] \, m \, dx. 
\]
Observe that if $\nabla_A \eta(x) \neq 0$, then $\frac{S_1}2 - 1 \le \bar\rho(x,E_\ell) \le \frac{S_1}2$.  Furthermore, since the distance from $E_\ell$ to $E_{\ell'}$,
$\ell'\neq \ell$,  is greater than $S_1$,  $\bar\rho(x,E_{\ell'}) \ge S_1/2$. 
Thus, since $E(\bar \mu + \delta) =  \bigcup_{\ell = 1}^R E_\ell$, we have
$\bar\rho(x,E(\bar \mu + \delta) ) \ge S_1/2 - 1$.  In particular, for any 
set $K$, $\nabla_A \eta(x)\neq 0$ implies
\[
h_K(x) = \bar\rho(x,E_{\bar\mu + \delta}\sm K) \ge \frac{S_1}2 - 1.
\]
We use this, \eqref{eq:expdecay.phi} with $K = M\setminus \Omega_\ell$, 
and $|\nabla_A \eta|^2 \le \bar V$ to obtain (recall the normalization $\|\vp\|=1$)
\begin{align*}
|r&(\zeta)| 
\le 
(\sup |\nabla_A \eta|)\|\nabla_A \zeta\|  
\Big(\int_{\{\nabla_A \eta\neq 0\}} \vp^2\, m \, dx  \Big)^{1/2} 
\\ &\hskip 4.6cm 
+(\sup |\nabla_A \eta|)\|\zeta\|  
\Big(\int_{\{\nabla_A \eta\neq 0\}} |\nabla_A \vp|^2\, m \, dx  \Big)^{1/2}  \\
&
\le
\|\nabla_A \zeta\|  
\Big(\int_{\{\nabla_A \eta\neq 0\}} \bar V\vp^2 \, m \, dx  \Big)^{1/2} 
+ 
\bar V^{1/2}\|\zeta\|  
\Big(\int_{\{\nabla_A \eta\neq 0\}} |\nabla_A \vp|^2\, m \, dx  \Big)^{1/2}  
\\
& \le 
(\|\nabla_A \zeta\|^2+ \bar V \|\zeta\|^2)^{1/2} 
\left(\frac{18e^2 \bar V^2}{\delta e^{S_1/2}}\right)^{1/2}.
\end{align*}
We will abbreviate this inequality by 
\begin{equation} \label{eq:error.r}
 r(\zeta)^2  
\le \eps \bar V [\|\nabla_A\zeta \|^2 + \bar V \|\zeta\|^2 ], \quad 
\eps : = 18e^2 \frac{\bar V}{\delta} e^{-S_1/2}.
\end{equation}

Since $V(x) \ge 0$, 
\[
\|\nabla_A\psi_j\|^2 + \bar V \|\psi_j\|^2 \le
\int_M [|\nabla_A\psi_j|^2 + (V + \bar V) \psi_j^2]\, m \, dx  
 = \lambda_{j} + \bar V.
\]
Let $J$ be any finite list of indices $j$ such that $|\lambda_j -\mu| \ge \delta $ and let
\[
\zeta = \sum_{j\in J} \gamma_j \psi_j 
\]
be any linear combination of the $\psi_j$. By density considerations, such a $\zeta$ is admissible. Then, since $V\ge 0$, 
\[
\|\nabla_A\zeta\|^2 + \bar V \|\zeta \|^2 
\le \int_M [|\nabla_A\zeta |^2 + (V + \bar V) \zeta^2] \, m \, dx
= \sum_{j\in J} (\lambda_j + \bar V)\gamma_j^2 \, .
\]
Consequently, it follows from \eqref{eq:error.r} that
\[
 r(\zeta)^2  
\le \eps \bar V [\|\nabla_A\zeta \|^2 + \bar V \|\zeta\|^2 ]
\le  \eps \bar V\sum_{j\in J} (\lambda_j + \bar V)\gamma_j^2 \, .
\]
Denote by 
\[
\beta_j = \int_M \eta\vp \psi_j \, m \, dx = \langle \eta \vp, \psi_j \rangle
\]
the coefficients of $\eta\vp$ in the basis. Because $(L -\lambda_j)\psi_j = 0$ in the weak sense,
\begin{align*}
r(\zeta)  
&= 
\sum_{j\in J} \gamma_j \int_M[\nabla_A\psi_j \nabla_A (\eta\vp) + (V-\lambda_j)\psi_j \eta \vp]\, m \, dx  
\\
& \hspace{130pt} + \int_M \gamma_j(\lambda_j-\mu) \psi_j \eta \vp \, m \, dx  
\\
& 
=\sum_{j\in J}  \gamma_j (\lambda_j-\mu) \beta_j \, .
\end{align*}
Thus,
\[
\Big|\sum_{j\in J}  \gamma_j (\lambda_j-\mu) \beta_j \Big|^2 = r(\zeta)^2 
\le  \eps \bar V\sum_{j\in J} (\lambda_j + \bar V)\gamma_j^2\, .
\]
Setting $\gamma_j = \beta_j (\lambda_j + \bar V)^{-1/2} \mbox{sgn}(\lambda_j - \mu)$, we find that 
\[
\left|\sum_{j\in J}  
\frac{|\lambda_j-\mu|}{\sqrt{\lambda_j + \bar V}} \beta_j^2 \right|^2 
\le  \eps \bar V\sum_{j\in J} \beta_j^2\, .
\]
Since $\lambda_j \ge0$ and $|\lambda_j -\mu| \ge \delta $, 
\[
\frac{|\lambda_j-\mu|}{\sqrt{\lambda_j + \bar V}} \ge \frac{\delta}{\sqrt{\bar V}} \, .
\]
Therefore,
\[
\sum_{j\in J} \beta_j^2 \le \eps \frac{\bar V^2}{\delta^2} .
\]
Since the set $J$ is an arbitrary finite subset of $j$ such that $|\lambda_j -\mu|\ge \delta$, we have 
\[
\| \eta \vp - \Psi_{(\mu -\delta, \mu  + \delta)} (\eta\vp)\|^2 \le \eps \frac{\bar V^2}{\delta^2}  \, .
\]
Next,  it follows from \eqref{eq:expdecay.phi} and $1-\eta(x) = 0$ on $\{\,\bar \rho(x,E_\ell) \le S_1/2 - 1\,\}$  that 
\[
\bar V\|(1-\eta)\vp\|^2  
\le \bar V \int_{h_K \ge \frac{S_1}2 -1} \, \vp^2 \, m \, dx  
\le \eps \bar V \quad (K = M\setminus \Omega_\ell),
\]
which, since the projection $I- \Psi_{(\mu -\delta, \mu  + \delta)}$ has operator norm $1$, implies that
\[
\| (1-\eta) \vp - \Psi_{(\mu -\delta, \mu  + \delta)} ((1-\eta)\vp)\|^2 
 \le  \eps.
\]
Finally, adding the bounds for $\vp = (1-\eta)\vp + \eta\vp$ and using $\delta \le \bar V/10$, we get
\[
\| \vp - \Psi_{(\mu -\delta, \mu  + \delta)} \vp\|^2 
\le 2 \eps \frac{\bar V^2}{\delta^2}
+ 2\eps   < 300 \frac{\bar V^3}{\delta^3} e^{-\bar S/2} \, .
\]
This is the first claim of the theorem (recall the normalization $\|\vp\| = 1$).

The second claim has a similar proof with the roles of $\vp$ and $\psi$ reversed. We will sketch each step, but the reader will need to refer regularly to the previous proof. Let $\psi $ be a normalized eigenfunction of $L$ on $M$ with eigenvalue $\lambda \le \bar \mu$. We use the same cutoff functions 
\[
\eta_\ell(x) = f(\bar \rho(x, E_\ell)),
\]
introducing the subscript $\ell$ since $\ell$ is no longer fixed. Then define
\[
\tilde \eta = \sum_\ell \eta_\ell.
\]
Note that $\tilde\eta \psi$ is compactly supported in the union of the $\Omega_\ell$, and the $\Omega_\ell$ are disjoint. Define the distribution $\tilde r$ by the equation
\[
L(\tilde \eta \psi) = \lambda \tilde \eta \psi + \tilde r.
\]
By similar reasoning to the proof of the first claim, using \eqref{eq:expdecay.psi} we have the analogue of \eqref{eq:error.r}, that for all $\zeta \in W^{1,2}(M)$, 
\[
\tilde r(\zeta)^2 \le \eps \bar V [\|\nabla_A \zeta\|^2 + \bar V\| \zeta\|^2],
\quad
\eps = 18 e^2 \frac{\bar V}{\delta} e^{-S_1/2}.
\]
Take any finite set $\tilde J$ of indices $(\ell, j)$ and denote
\[
\tilde \zeta = \sum_{(\ell,j)\in \tilde J} \gamma_{\ell j} \vp_{\ell j} \, .
\]
In the same way as before, we deduce
\[
\tilde r(\tilde \zeta)^2 
\le \eps \bar V \sum_{\tilde J} (\mu_{\ell j} + \bar V) \gamma_{\ell j}^2.
\]
Moreover, as before, if we define
\[
\beta_{\ell j} = \int_M \tilde \eta\psi \vp_{\ell j} \, m \, dx.
\]
We claim that 
\[
\tilde r (\tilde \zeta) = \sum_{\tilde J} \gamma_{\ell j} (\mu_{\ell j} - \lambda) \beta_{\ell j}.
\]
This last identity is the only place where the proof is slightly different. Observe that because $(L - \mu_{\ell j} )\vp_{\ell j}=0$ in the weak sense on $\Omega_\ell$ and $\eta_{\ell'}$ has support disjoint from $\bar \Omega_\ell$ for all $\ell'\neq \ell$, 
\begin{align*}
\int_M 
[\nabla_A (\tilde \eta \psi)  \nabla_A \vp_{\ell j} + & 
(V-\mu_{\ell j})\tilde \eta \psi \vp_{\ell j}]
\, m \, dx  \\
& =
\int_{\Omega_\ell} 
[\nabla_A ( \eta_\ell \psi) \nabla_A \vp_{\ell j} + (V-\mu_{\ell j})\eta_\ell \psi \vp_{\ell j}] 
\, m \, dx  = 0.
\end{align*}
This is the only aspect of the proof of the formula for $\tilde r(\tilde \zeta)$ that differs from the one for $r(\zeta)$ above.

Now suppose that for every $(\ell,j)\in \tilde J$, $|\mu_{\ell j} - \lambda| \ge \delta $. Then, setting 
\[
\gamma_{\ell j} = \beta_{\ell j} (\mu_{\ell j} +\bar V)^{1/2} \mbox{sgn}(\mu_{\ell j} - \lambda),
\]
we obtain 
\[
\sum_{\tilde J} \beta_{\ell j}^2 \le \eps \frac{\bar V^2}{\delta^2}.
\]
Since $\tilde\eta \psi$ is supported in the union $\displaystyle \bigcup_\ell  \Omega_\ell$ and the $\vp_{\ell j}$ are an orthonormal basis for $L^2$ on that set, and $L$ is an arbitrary finite subset of indices such that $|\mu_{\ell j} - \lambda| \ge \delta$, we have
\[
\| \tilde\eta \psi 
- \Phi_{(\lambda-\delta, \lambda + \delta)} 
(\tilde\eta \psi )\|^2 
\le \eps \frac{\bar V^2}{\delta^2}.
\]
Next, it follows from the fact that $(1-\tilde \eta(x)) = 0$ on the set where $h(x)  = \bar \rho(x,E(\bar \mu + \delta)) \le S_1/2 - 1$ and \eqref{eq:expdecay.psi} that 
\[
\bar V\|(1-\tilde \eta)\psi\|^2  \le \bar V \int_{h \ge \frac{S_1}2 -1} \, \psi^2 
\, m \, dx  \le \eps \bar V.
\]
The rest of the proof is similar.
\end{proof}

Theorem~\ref{thm:spectralprojection} shows that when the landscape potential $1/u(x)$ defines wells that are separated by a large number $S$, then the eigenfunctions are located in these wells (with a single eigenfunction possibly occupying several wells). An easy consequence is the following corollary saying that the graphs of the two counting functions enumerating eigenvalues of $L$ and eigenvalues localized to wells agree (modulo a shift $\pm \delta$) up to a number $\bar N$ defined below.
\begin{corollary} \label{C4.3}
Consider the counting functions
\[
N(\lambda) = \# \{ \lambda_j:\lambda_j\le \lambda\};
\quad
N_0(\mu) = \# \{ \mu_{\ell,j} :  \mu_{\ell,j}  \le \mu\}.
\]
Recall that $\bar \mu$ and $\delta$ are used to specify $\bar S$. Suppose that $\mu\le \bar \mu$ and choose $\bar N$ such that 
\[
300 \bar N \left(\frac {\bar V}{\delta}\right)^3 <  e^{\bar S/2}. 
\]
Then 
\[
\min(\bar N, N_0(\mu - \delta)) \le N(\mu) \quad \mbox{and}
\quad 
\min(\bar N, N(\mu-\delta)) \le N_0(\mu) .
\]
\end{corollary}

\begin{proof}  Let 
\[
p = \min(\bar N, N(\mu-\delta))
\]
Consider the first $p$ eigenvectors $\psi_1$, \dots $\psi_p$ of $L$ on $M$. Then  $p\le N(\mu -\delta)$ implies $\lambda_j \le \mu -\delta$, and therefore
\[
\| \psi_j - \Phi_{(0,\mu)} \psi_j\|^2 \le  300 \left(\frac{\bar V^3}{\delta^3}\right) e^{-\bar S/2}.
\]
For any nonzero linear combination $\psi = \sum_{j=1}^p \alpha_j\psi_j $, we have
\begin{align*}
\| \psi - \Phi_{(0,\mu)}\psi\| 
&\leq \sum_j |\alpha_j|  \| \psi_j - \Phi_{(0,\mu)} \psi_j\|
\\
&\leq \Big(300 \left(\frac{\bar V^3}{\delta^3}\right) e^{-\bar S/2} \Big)^{1/2} \sum_j |\alpha_j|
\\
&\le \Big(300 \left(\frac{\bar V^3}{\delta^3}\right) e^{-\bar S/2} \Big)^{1/2} \|\psi\| p^{1/2} < \|\psi \|,
\end{align*}
by the Cauchy-Schwarz inequality and because $p\le \bar N$. Denote by $Q$ the span of the $\psi_j$, $j=1,\, \dots, \,  p$. The inequality implies the restriction of $\Phi_{(0,\mu)}$ to $Q$ is injective and the dimension $N_0(\mu)$ of $\Phi_{(0,\mu)}(Q)$ is at least $p$. In other words, $N_0(\mu) \ge p$. The proof of the lower bound for $N(\mu)$ is similar.
\end{proof}

\section{Manifolds and approximation}\label{sec:manifolds}

In this section we discuss two generalizations of the results of Section~3: the extension to manifolds and the removal of the continuity assumption on the coefficients of $A$. We also prove the boundary regularity for mixed data referred to in Section~3 and construct
the intermediate sets $K_\ell$ of \eqref{4.1a} with the bi-Lipschitz cone condition.

Let us first see how to replace $\RR^n$ with an ambient space $\widehat M$ defined as a compact, connected  $C^1 $ manifold. Let $V$ be a  bounded measurable function satisfying $0 \le V(x) \le \bar V$ on $\widehat M$. Let $A$ be a symmetric two-tensor and let $m$ be a density on $\widehat M$. In a coordinate chart, $x$,  $A$ is represented locally by a symmetric matrix-valued function (which we shall still denote by $A$) and $m$ is represented by a scalar function. Given a test function $\eta=\eta(x)$ compactly supported in the coordinate chart, and a function $\vp=\vp(x)$, we write
\[
\langle A \nabla \vp, \nabla \eta\rangle 
:= \int (A \nabla \vp)\cdot \nabla \eta \, m\, dx,
\quad  \langle \vp,\eta\rangle := \int \vp \, \eta \, m \, dx .
\]
We extend these definitions to test functions on all of $\widehat M$ by using a partition of unity. The covariance property that makes this definition  independent of the choice of coordinate charts is that in a new coordinate system $y$ with $x= x(y)$, the expression for the corresponding matrix $ \tilde A (y) $ and density $\tilde m(y) $ is
\[
\tilde A (y) 
 = B(y)^{-1} A(x(y))  (B(y)^{-1})^T,
\quad  \tilde m(y) 
= |\det B| m(x(y)),
\]
where $B$ is the Jacobian matrix
\[
B_{ij}(y) = \frac{\partial x_i}{\partial y_j}, \quad B = (B_{ij}).
\]
For $\eta $ supported in the intersection (in the $x$ variable) of the two coordinate charts, denoting $\tilde \eta(y) = \eta(x(y))$, $\tilde \vp(y) = \vp(x(y))$, and $\tilde V(y) =V (x(y))$, we have 
\[
\int [(A\nabla \vp)\cdot \nabla \eta +
 V  \vp  \eta ] \,m \, dx 
=\int [(  \tilde A 
\nabla \tilde \vp) \cdot  \nabla \tilde \eta + \tilde V\tilde \vp \tilde \eta] \, \tilde m \, dy.
\]
Thus we obtain globally defined quantities $\langle A\nabla \vp,\nabla \eta\rangle$ and $\langle V\vp, \eta\rangle$. 

We will assume that in some family of coordinate charts covering all of $\widehat M$, $A$ is represented by bounded measurable, uniformly elliptic matrices and that $m$ is bounded above and below by positive constants. The constant of ellipticity and the constants bounding $m$ from above and below depend on the coordinate charts. But since our estimates won't depend on these constants, this does not matter to us.

Let $\Omega$ be an open, connected subset of $ \widehat M$ such that near each point of $\partial \Omega$, $\Omega$ is locally bi-Lipschitz equivalent to a half space. This includes as a special case, bi-Lipschitz images of Lipschitz domains in $\RR^n$ (for instance, bounded chord-arc domains in $\RR^2$). It also includes the case $\Omega = \widehat M$ in which the boundary is empty. Set $M = \bar \Omega$. Denote the inner product associated to $L^2(M)$ with density $m$ by $\langle \,\cdot\,,\,\cdot\,\rangle$. Let $K$ be a compact subset of $M$ and let $W^{1,2}_0(M\setminus K)$ denote the closure in $W^{1,2}$ norm of the set of functions in $C^1(M)$ that vanish on $K$. For $\vp\in W_0^{1,2}(M\setminus K)$ and  $f \in L^2(M\setminus K)$, the weak equation $L \vp= f$ on $M\setminus K$ is defined by
\[
\langle A \nabla \vp, \nabla \eta\rangle  + \langle V\vp, \eta\rangle = \langle f,\eta\rangle
\]
for every $\eta \in W_0^{1,2}(M \setminus K)$.

We will now prove H\"older regularity of solutions up to the boundary for suitable 
$K$ and $f$. 
\begin{proposition}\label{prop:boundaryHolder}  Suppose that
$\Omega\subset \wh M$ is locally bi-Lipschitz equivalent to a half space
at each boundary point.  
Suppose that $K$ satisfies the bi-Lipschitz cone condition
as defined above Proposition~\ref{prop:phi.j.Holder}. There is $\alpha>0$ such that if  $f\in L^\infty(M)$ and $\vp \in W^{1,2}(M)$, with $\vp =0$ on $K$, solves $(L-\mu)\vp = f$ in the weak sense on $M\setminus K$, then $\vp\in C^\alpha(M)$.  
\end{proposition}
\begin{proof}  Without loss of generality, we can replace $V$ by $V-\mu$ and 
assume the constant $\mu=0$. 
As we have already observed, the interior H\"older regularity follows from the theorem of De Giorgi-Nash-Moser.  We handle the Neumann boundary conditions by
performing an even reflection at the boundary of $\Omega$.

It will suffice to consider a single coordinate chart denoted here by $y$. Let 
\[
B_r= \{y\in \RR^n: |y|<r\}, \quad Q =\{y\in \RR^n:  y_1\ge 0\}.
\]
and let $K$ be a compact subset of $\bar B_1\cap Q$ satisfying the bi-Lipschitz cone
condition.   Let $W^{1,2}((B_1\cap Q)\setminus K)$ 
be the closure in 
$W^{1,2}$ norm of functions of $C^1(\bar B_1 \cap Q)$ with
support disjoint from $K$.   If $\vp \in W^{1,2}((B_1\cap Q)\setminus K)$, then the
extension of $\vp$ by $0$ on $K$ belongs to $W^{1,2}(B_1\cap Q)$. 
For $f\in L^\infty((B_1\cap Q))$ we say $\vp$ solves $L\vp = f$ 
weakly on $(B_1\cap Q)\setminus K$ if 
\[
\int_{B_1\cap Q} [(A  \nabla \vp) \cdot \nabla \eta  + V \vp \eta] \, m \, dy
= \int_{B_1\cap Q} f  \eta  \, m \, dy
\]
for all $\eta\in C^1(B_1\cap Q)$ with support disjoint from $K$.   
(The fact that $\eta$ need not vanish on  $y_1=0$ is what imposes the Neumann condition in the weak sense.) Here, as usual, $A$ is a bounded measurable symmetric matrix,  $f$, $V$ and $m$ bounded measurable functions
defined in $B_1\cap Q$.  Moreover, $A$ is elliptic (see \eqref{eq:ellipticity}) and $1/C \le m(y) \le C$.  

We extend $m$, $V$, $\vp$ and $f$ to $B_1$ by reflection as follows.
Let $R$ be the reflection,
\[
R(y_1,y_2,\dots, y_n) = (-y_1,y_2,\dots, y_n).
\]
Set $\tilde m(y) = m(y)$, $\tilde V(y) = V(y)$, 
$\tilde \vp(y) = \vp(y)$, 
$\tilde f(y) = f(y)$,  for $y\in B_1\cap Q$, and 
\[
\tilde m(y) = \tilde m(Ry), \quad 
\tilde V(y) = \tilde V(Ry), \quad \tilde \vp(y) = \tilde \vp(Ry), \quad  \tilde f(y) = \tilde f(Ry).
\]
Define $\tilde K =  K \cup RK$, then $\tilde \vp = 0$ on $\tilde K$. 
We extend $A$ to $B_1$ by 
\[
\tilde A(y) = R \tilde A(R y) R.
\]
Note that this is just the appropriate covariance for the changes of variable $R$ since 
$R = R^T = R^{-1}$. In this way, we extend the definition of $L$ to an operator $\tilde L$ on $B_1$. 

We claim that $\tilde L \tilde \vp = \tilde f$ weakly on $B_1\setminus \tilde K$.
To prove this, let $\eta \in C^1(\bar B_1) $ be such that the support of $\eta$
is disjoint from $\tilde K \cup \partial B_1$.  Denote 
\[
\eta_*(y) = \frac12( \eta(y) + \eta(Ry)),  \quad y\in B_1. 
\]
Observe that the $*$ operation symmetrizes $\eta$, whereas $\tilde \vp$ and $\tilde f$ are defined so that they have this symmetry already: $\tilde \vp_* = \tilde \vp$  and $\tilde f_* = \tilde f$.

Denote the inner products on $L^2(B_1, \tilde m\, dx)$ and $L^2(B_1 \cap Q, m \, dx)$ 
by $\langle \,\cdot\, , \, \cdot\, \rangle_{B_1}$ and $\langle \,\cdot\, , \, \cdot\, \rangle_{Q}$, respectively. Since $\tilde \vp = \tilde \vp_*$, 
\[
\langle \tilde A 
\nabla \tilde \vp, \nabla \eta\rangle_{B_1} 
=
\langle  \tilde A 
\nabla \tilde \vp_*, \nabla \eta\rangle_{B_1} 
=
\langle  \tilde A 
\nabla \tilde \vp, \nabla \eta_*\rangle_{B_1}
=
2\langle   A 
\nabla \vp, \nabla \eta_*\rangle_{Q}
\]
Furthermore, using the fact that $\eta_*(y)=0$ on $K$ 
and the weak equation for $\vp$ on $B_1\cap Q$,  we have
\[
2 \langle  A \nabla  \vp, \nabla \eta_*\rangle_{Q}
=
2 \langle  f,  \eta_*\rangle_{Q} - 2\langle V\vp,   \eta_*\rangle_{Q}
= 
\langle  \tilde f,  \eta \rangle_{B_1} - \langle  \tilde V
\tilde \vp,   \eta \rangle_{B_1}.
\]
Combining these two equations, 
\[
\langle  \tilde A 
\nabla \tilde \vp, \nabla \eta\rangle_{B_1} 
+ \langle  \tilde V 
\tilde \vp,   \eta \rangle_{B_1} = 
\langle  \tilde f,  \eta \rangle_{B_1}.
\]
In other words, $\tilde L \tilde \vp = \tilde f$ weakly on $B_1\setminus \tilde K$, which
was what we claimed. 

We are now in a position to quote local boundary regularity theorems
of Gilbarg and Trudinger.   Theorems  8.25 and 8.26 of \cite{GT} imply 
that
\[
\sup_{B_{1/2}} |\vp|  \le C (\| \vp\|_{L^1(B_1)} + \|f\|_{L^\infty(B_1)}) 
\]
with a constant $C$ depending only on the ellipticity constants.
(Note that the appropriate notion of supremum for $W^{1,2}$ functions,
based on truncation, is defined just before Theorem 8.25.)

Next, the local Dirichlet boundary regularity theorem, Theorem 8.27 \cite{GT},
implies that since $K$ satisfies
the bi-Lipschitz cone
condition,\footnote{To apply the theorem as stated one has to 
make a bi-Lipschitz change of variables to produce an actual cone.
This changes the ellipticity constant by a fixed factor. 
There is an additional term in the estimate in Theorem 8.27, namely,
the oscillation of $\vp$ over $K\cap B_{\sqrt r}$. But in our case, this is zero.} 
there is $\alpha>0$ such that for $r \le 1/4$, 
\[
\underset{B_{r}}{\mbox{osc}}\ 
 \vp \le C\, r^\alpha \sup_{B_{1/2}}|\vp|.
\]
This proves H\"older continuity up to the boundary.
\end{proof}

\begin{lemma}\label{lem:cubecover}
Let $K$ be a compact subset of $M$. Let $U$ be a (relatively) open set in $M$ 
such that $K\subset U$. Then there is 
a compact set $K'$, such that $K\subset K' \subset U$, 
$K'$  satisfies the bi-Lipschitz cone condition.
\end{lemma}

\begin{proof}   
To find $K'$ given $K$, cover  $M$ with a finite number of coordinate 
charts each of which is the bi-Lipschitz image of a closed cube, 
some of them interior to $\Omega$
and others with a boundary face on $\partial \Omega$.   Fix $\eps>0$,
and  subdivide each closed cube of the covering dyadically to get
a finite covering by cubes of diameter less than $\eps$.  Note
that although this is not a disjoint covering because of the overlap of 
the coordinate charts, it is a finite covering.  Define $K'$ as the union of cubes in
the subdivision that intersect  $K$.  

For $\eps$ sufficiently small $K'\subset U$.  Each individual
bi-Lipschitz cube satisfies the bi-Lipschitz cone condition, 
so this finite union also satisfies the condition
\end{proof}

\bigskip

The last difficulty that we wish to address is that the Agmon length of paths is not defined for discontinuous $A$.  Suppose that $A$ is bounded and measurable
(and symmetric and uniformly elliptic as in \eqref{eq:ellipticity}).   Using
convolution on coordinate charts and a partition of unity, we find a sequence of $A^\eps$ of continuous uniformly elliptic two-tensors such that $A^\eps$ tends pointwise to $A$ as $\eps \to 0$. Denote by $L$ and $L_\eps$ the operators on $M$ corresponding formally in local coordinates to $-(1/m)\div(m A\nabla) + V$ and $-(1/m)\div(mA^\eps \nabla) + V$.  

\begin{proposition} \label{prop:approx.A}  
Let $\lambda_\eps$ be a bounded sequence, and suppose that $L_\eps \psi_\eps = \lambda_\eps \psi_\eps$ in the weak sense on $M$, and normalize the eigenfunctions by $\|\psi_\eps\| = 1$ in $L^2(M)$. Then there is a subsequence $\eps_j\to 0$ such that  
\begin{enumerate}[a)]
\item $\psi_{\eps_j}$  has a limit $\psi$ in $W^{1,2}(M)$ norm and in  $C^\alpha(M)$ norm for some $\alpha>0$.
\item $\lambda_{\eps_j}$ has a limit $\lambda$ and $L\psi = \lambda\psi$ in the weak sense on $M$.
\end{enumerate}
\end{proposition}
\begin{proof}
By the nondegeneracy of $V$, the sequence $\psi_\eps$ is uniformly bounded in $W^{1,2}(M)$ norm. Moreover by de Giorgi-Nash-Moser regularity the sequence is bounded in $C^\beta(M)$ norm for some $\beta>0$. Note that $\beta$ can be chosen independently of $\eps$ because ellipticity constants of $A^\eps$ are uniformly controlled. By the compactness of $C^\beta(M)$ in $C^\alpha(M)$ for $\alpha < \beta$ and the weak compactness of the unit ball of $W^{1,2}(M)$, there is a subsequence $\eps_j\to 0$ such that $\psi_{\eps_j}$ converges in $C^\alpha(M)$ norm to a function $\psi\in C^\alpha(M) \cap W^{1,2}(M)$. Moreover, $\nabla \psi_{\eps_j} \to \nabla \psi$  weakly in $L^2(M)$ and $\lambda_{\eps_j} \to \lambda$ as $j\to \infty$. Hence, taking the weak limit in the equation $L_\eps \psi_\eps = \lambda_\eps\psi_\eps$, we obtain, $L \psi  = \lambda \psi$.  

It remains to show that $\nabla \psi_{\eps_j}$ tends to $\nabla \psi$ in $L^2(M)$ norm.  Indeed, by the dominated convergence theorem,
\begin{equation}\label{5.3a}
\|(A^{\eps_j}-A)\nabla \psi\|  \to 0 \quad \mbox{as}\ j\to \infty.
\end{equation}
From now on, we will omit the subscript $j$ from $\eps$ with the understanding that we have passed to a subsequence of the $A^\eps$ and the $\psi_\eps$. It follows that, along this subsequence,
\[
\langle (A-A^\eps) \nabla\psi ,\nabla \psi\rangle \to 0 \ 
\hbox{ and }  
\ \langle A^\eps \nabla\psi ,\nabla \psi\rangle \to
\langle A \nabla\psi ,\nabla \psi\rangle.
\]
Furthermore, since $\|\nabla \psi_\eps\|$ is uniformly bounded and by \eqref{5.3a}, 
\[
\langle (A-A^\eps) \nabla\psi ,\nabla \psi_\eps\rangle \to 0 .
\]
This combined with the weak limit 
$ \langle A \nabla\psi ,\nabla \psi_\eps\rangle \to 
 \langle A \nabla\psi ,\nabla \psi \rangle $
 yields
 \[
 \langle A^\eps\nabla\psi ,\nabla \psi_\eps\rangle \to   
 \langle A \nabla\psi ,\nabla \psi \rangle.
 \]
Using  the identity $L_\eps \psi_\eps = \lambda_\eps \psi_\eps$, we write
\[
\langle A^\eps \nabla \psi_\eps, \nabla \psi_\eps \rangle = \lambda_\eps - \langle V\psi_\eps,\psi_\eps\rangle \to \lambda - \langle V\psi, \psi\rangle.
\]
Finally, 
\[
\langle A^\eps \nabla (\psi_\eps-\psi),\nabla(\psi_\eps-\psi)\rangle
= \langle A^\eps \nabla \psi_\eps ,\nabla \psi_\eps\rangle
-2 \langle A^\eps \nabla \psi,\nabla \psi_\eps\rangle
+\langle A^\eps \nabla \psi,\nabla \psi\rangle.
\]
The first term of this last expression, 
$\langle A^\eps \nabla \psi_\eps ,\nabla \psi_\eps\rangle \to 
\lambda - \langle V\psi, \psi\rangle$.  The second term tends to 
$-2\langle A\nabla \psi,\nabla\psi\rangle$ and the third term to 
$\langle A\nabla \psi,\nabla\psi\rangle$.   But $L\psi = \lambda \psi$
implies $\langle A\nabla \psi,\nabla\psi\rangle = \lambda - \langle V\psi, \psi\rangle$.
Thus 
\[
\langle A^\eps \nabla (\psi_\eps-\psi),\nabla(\psi_\eps-\psi)\rangle \to 0
\]
along the subsequence and $\nabla\psi_\eps$ tends in $L^2(M)$ norm to $\nabla \psi$. 
\end{proof}

Let $A$ have bounded measurable coefficients and let $A^\eps$ be a continuous approximation as above. Then the compactness argument in the proposition also shows that the landscape function $u_\eps$ tends uniformly to the landscape function $u$ along a suitable subsequence. Because the Agmon distance functions are uniformly Lipschitz, at the expense of a further subsequence, one can ensure that this distance also converges uniformly. Notice that different sequences could, in principle, yield different limiting Agmon distances. For any of the limits we can now deduce estimates analogous to the ones in the previous sections.

We illustrate with \eqref{eq:expdecay.psi} and discuss the subsequent theorems later. Fix $\mu$ and let $W_\mu$ be the subspace of $L^2(M)$ spanned by the eigenfunctions of $L$ with eigenvalue $\le \mu$, and let $N$ be the dimension of $W_\mu$. Denote
\[
\mu_\eps = \sup_{\psi\in W_\mu} 
\frac{\langle L_\eps \psi ,\psi\rangle}{\langle \psi,\psi\rangle}.
\]
Let $\psi_j^\eps$, $j=1,\, \dots, \, N$ be the first $N$ eigenfunctions of $L_\eps$, and let $\lambda_j^\eps$ be the corresponding eigenvalues. It follows from the min/max principle and the fact that $W_\mu$ has dimension $N$ that $\lambda_j^\eps \le \mu_\eps$, $ j \le N$.

We claim that 
\begin{equation}\label{5.4a}
\limsup_{\eps \to 0} \mu_\eps \le \mu . 
\end{equation}
In fact, if $\psi_j$ satisfying $L \psi_j = \lambda_j\psi_j$, $j=1,\, \dots,\, N,$ is an orthonormal basis of $W_\mu$, then by the dominated convergence theorem, for every $\delta>0$ there is $\eps_0>0$ such that for $\eps < \eps_0$, 
\[
|\langle L_\eps \psi_j, \psi_k \rangle - \delta_{jk} \lambda_j| \le \delta.
\]
Representing $\psi$ as a linear combination of the $\psi_j$, we deduce from $\lambda_j\le \mu$ that $\mu_\eps \le \mu + N^2 \delta$. Hence \eqref{5.4a} holds.

By Proposition~\ref{prop:approx.A}, for a suitable subsequence of values of $\eps$  the orthonormal basis $\psi_j^\eps$, $j\le N$, tends in $C^\alpha(M)$ and $W^{1,2}(M)$ norm to an orthonormal set of eigenfunctions of $L$ with eigenvalues $\le \mu$. Since $W_\mu$ has dimension $N$, this limiting set must be a basis for $W_\mu$. Moreover, these eigenfunctions inherit the inequality \eqref{eq:expdecay.psi}.

There is a difference between this statement and the preceding one, applicable to continuous $A$. Here we only claim that there exists a basis of the eigenfunctions that satisfies \eqref{eq:expdecay.psi}. If an eigenvalue has multiplicity then the estimate may not apply to all linear combinations of the particular eigenbasis we obtain by taking limits. Thus, we have not ruled out the possibility that there has to be an extra factor of the multiplicity of the eigenspace in inequality \eqref{eq:expdecay.psi}. Similarly, in the comparisons with localized eigenfunctions in Theorem~\ref{thm:spectralprojection}, we can only deduce that they are valid for {\em some} basis of eigenfunctions $\psi_j$ and $\vp_{\ell,j}$.  

We leave open whether in the case of discontinuous $A$, it is possible to recover the full theorem for continuous coefficients for eigenfunctions with multiplicity. Another question that we are leaving open in the discontinuous case is whether the limiting Agmon distance is unique, that is, does not depend on the choice of the sequence $A^\eps$. Even if the limit is not unique, there could be an optimal (largest) choice of $h$ satisfying the Agmon bound $|\nabla_A h|^2 \le w_\mu(x)$.

\end{document}